\documentclass{amsart}
\usepackage{amsmath,amsfonts,amssymb,amsthm,amscd,latexsym,euscript,xypic, verbatim}
\usepackage[all]{xy}
\usepackage{graphics}
\usepackage{epsfig}
\usepackage{multirow}
\theoremstyle{plain}

\newtheorem*{thmA}{Theorem A}
\newtheorem*{thmB}{Theorem B}
\newtheorem*{thmC}{Theorem C}
\newtheorem*{thmD}{Theorem D}

\newtheorem*{thm1.2}{(1.2) Theorem}
\newtheorem*{thm1.3}{(1.3) Theorem}
\newtheorem*{thm1.4}{(1.4) Theorem}
\newtheorem*{propA*}{Proposition A}
\newtheorem*{propB*}{Proposition B}
\newtheorem*{thmC*}{Theorem C}
\newtheorem*{propD*}{Proposition D}
\newtheorem{prop}{Proposition}[section]
\newtheorem{thm}[prop]{Theorem}
\newtheorem{cor}[prop]{Corollary}
\newtheorem{lemma}[prop]{Lemma}

\theoremstyle{definition}

\newtheorem{Def}[prop]{Definition}
\newtheorem*{Def*}{Definition}

\newtheorem*{notation*}{Notation}
\newtheorem*{question*}{Question}

\def\dmono{\dto|<\hole|<<\ahook}
\def\rmono{\rto|<\hole|<<\ahook}

\def\repi{\rto|>>\tip}

\def\depi{\dto|>>\tip}

\newcommand{\ra}{\rightarrow}

\newcommand{\Aut}{{\rm Aut}}

\newcommand{\Ker}{{\rm Ker\,}}
\newcommand{\Hom}{{\rm Hom}}

\newcommand{\onto}{\twoheadrightarrow}

\newcommand{\cell}{\text{cell}}

\newcommand{\gensub}[1]{\text{\rm Sub} (#1)}
\newcommand{\itgensub}[2]{\text{\rm Sub}^{#2}(#1)}
\newcommand{\surgensub}[1]{\text{\rm SurSub}(#1)}
\newcommand{\itsurgensub}[2]{\text{\rm SurSub}^{#2}(#1)}
\newcommand{\cov}[1]{\text{\rm Cov}(#1)}
\newcommand{\itcov}[2]{\text{\rm Cov}^{#2}(#1)}
\newcommand{\surcov}[1]{\text{\rm SurCov}(#1)}
\newcommand{\itsurcov}[2]{\text{\rm SurCov}^{#2}(#1)}
\newcommand{\invquot}[1]{\text{\rm InvQuot} (#1)}
\newcommand{\invsub}[1]{\text{\rm InvSub}(#1)}
\numberwithin{equation}{section}
\begin{document}
\title[Idempotent deformations]{Idempotent deformations of  finite groups}
%\title[Value of idempotent functors]{On the value of idempotent functors on finite groups}
\author{M.~Blomgren}
\email{blomgr@kth.se}
\address{Martin Blomgren\\
       Department of Mathematics\\
       KTH\\
       S 10044 Stockholm\\
       Sweden}
\author{ W.~Chach\'{o}lski}
\email{wojtek@math.kth.se}
\address{Wojciech Chach\'{o}lski\\
       Department of Mathematics\\
       KTH\\
       S 10044 Stockholm\\
       Sweden}
\author{E.~D.~Farjoun}
\email{farjoun@math.huji.ac.il}
\address{Emmanuel D.~Farjoun\\
 Department of Mathematics\\
   Hebrew University of Jerusalem, Givat Ram\\
     Jerusalem 91904\\
     Israel}
\author{Y.~Segev}
\email{yoavs@math.bgu.ac.il}
\address{Yoav Segev\\
    Department of Mathematics\\
       Ben-Gurion University\\
       Beer-Sheva 84105\\
       Israel}
% \shortauthors{Blomgren, Chach\'{o}lski, Farjoun, Segev}
 \thanks{The second author partially supported by G\"oran Gustafsson Stiftelse and VR grants 2002-2005 and 2005-2008.}
  \thanks{The third and fourth authors partially supported by the Israel Science Foundation grant no.~712/07}
%\classification{55P99, 20D05, 19C09}
%\author[Blomgren, Chach\'{o}lski, Farjoun, Segev]
%{M.~Blomgren\quad  W.~Chach\'{o}lski$^1$\quad   E.~D.~Farjoun$^2$\quad Y.~Segev$^2$}
%\address{Martin Blomgren\\
%       Department of Mathematics\\
%       KTH\\
%       S 10044 Stockholm\\
%       Sweden}
%\email{blomgr@kth.se}

%\address{Wojciech Chach\'{o}lski\\
%       Department of Mathematics\\
%       KTH\\
%       S 10044 Stockholm\\
 %      Sweden}
%\email{wojtek@math.kth.se}

%\address{Emmanuel D.~Farjoun\\
 %      Department of Mathematics\\
 %      Hebrew University of Jerusalem, Givat Ram\\
 %      Jerusalem 91904\\
%       Israel}
%\email{farjoun@math.huji.ac.il}

%\address{Yoav Segev\\
%       Department of Mathematics\\
%       Ben-Gurion University\\
%       Beer-Sheva 84105\\
 %      Israel}
%\email{yoavs@math.bgu.ac.il}

%
%
%ABSTRACT
%
%
% \begin{abstract}
%\end{abstract}

%\subjclass[2010]{20D05}
%\subjclass[2010]{20D99}
%\subjclass[2010]{55U99}
%\date{December 26, 2010}
\maketitle

%\footnotetext[1]{Partially supported by G\"oran Gustafsson Stiftelse and VR grants 2002-2005 and 2005-2008.}
%\footnotetext[2]{Partially supported by the Israel Science Foundation grant no.~712/07}

%\renewcommand{\thefootnote}{}
%\footnote{\textit{Key words and phrases.}  idempotent functor, cellular cover,
%central extensions, Schur multiplier}
%\footnote{\textit{2000 Mathematics Subject Classification} Primary: 55P99, 20D05, 19C09}

%%%%%%%%%%%%%%%%%%%%%%%%%%%%%%%%%%%%%%%%%%%%%%%%%%%%
%%%%%%%%%%%%%%%%%%%%%%%%%%%%%%%%%%%%%%%%%%%%%%%%%%%
%%%%%%%%%%%%%%%%%%%%%%%%%%%%%%%%%%%%%%%%%%%%%%%%%%%
\section{Introduction}
%%%%%%%%%%%%%%%%%%%%%%%%%%%%%%%%%%%%%%%%%%%%%%%%%%%
%%%%%%%%%%%%%%%%%%%%%%%%%%%%%%%%%%%%%%%%%%%%%%%%%%
%%%%%%%%%%%%%%%%%%%%%%%%%%%%%%%%%%%%%%%%%%%%%%%%%%%%
One way to understand symmetry of some objects is to look for  what acts on them and study operations on these
objects. In this way  we study symmetry of groups by considering endofunctors $\phi:\text{Groups}\ra\text{Groups}$.  To understand how such an operation $\phi$
deforms  groups we consider natural transformations $\epsilon_G:\phi(G)\ra G$.
A choice of a functor $\phi:\text{Groups}\ra\text{Groups}$ and a natural transformation 
$\epsilon_G:\phi(G)\ra G$ is called an augmented functor and denoted by $(\phi,\epsilon)$.
By iterating  the augmentation we obtain  two homomorphisms $\epsilon_{\phi(G)}:\phi^2(G)\ra \phi(G)$
and $\phi(\epsilon_{G}):\phi^2(G)\ra \phi(G)$.  Among all augmented functors $(\phi, \epsilon)$
there are the idempotent ones for which this iteration process does not
produce anything new and the homomorphisms  $\epsilon_{\phi(G)}$ and  $\phi(\epsilon_{G})$ are isomorphisms for any group $G$.  The universal central
extension of the maximal perfect subgroup of
$G,$ with the natural projection as augmentation, is an  example of  an idempotent functor.
 Idempotent  functors  are related to the  concept of cellularity which was   introduced originally in
homotopy theory and has been used  to
organize information about spaces.
In recent years these functors have been considered in
 algebraic context of groups, chain complexes etc., see for example~\cite{A, Ca1, CaD, CaRoSce, CDFS, CFGS, F, FGS, FGSS, Fl, MP, RoSc}.

The main aim of this paper is to understand how idempotent functor deform  finite groups, particularly the 
simple ones.  Our first  result is
(see Corollary~\ref{cor idemfinfin}, where preservation of  nilpotency and solvability is  also discussed):

%%%%%%%%%%%%%%%%%%%%%%%%%%%%%%%%%%%%%%%%%%%%%%%%%%%%%%%%%%%%%%%%%%%%

%%%%%%%%%%%%%%%

%%%%%%%%%%%%%%%%%%%%%%%%%%%%%%%%%%%%%%%%%%%%%%%%%%%%%%%
\begin{thmA}
%%%%%%%%%%%%%%%%%%%%%%%%%%%%%%%%%%%%%%%%%%%%%%%%%%%%%%%%
Let  $(\phi, \epsilon)$ be an idempotent functor. If $G$ is finite, then so is $\phi(G)$.
\end{thmA}

In this way finite groups are acted upon by  idempotent functors.
 How complicated is this action?
To measure it, we  study the orbits of this action: 
\begin{Def}
$\text{\rm Idem}(G):=\{\text{isomorphism class of }\phi(G)\ |\ (\phi,\epsilon)\text{ is idempotent}\}.$
\end{Def}

Although the collection of  idempotent functors does  not even form a set,
 the  number of different values idempotent functors can take on
a given  finite group is finite (see Corollary~\ref{cor covidemfinite}):

%%%%%%%%%%%%%%%%%%%%%%%%%%%%%%%%%%%%%%%%%%%%%%%%%%%%%%%%%%%
\begin{thmB}
%%%%%%%%%%%%%%%%%%%%%%%%%%%%%%%%%%%%%%%%%%%%%%%%%%%%%%%%%%%%%%
If $G$ is a finite group, then $\text{\rm Idem}(G)$ is a finite set.
\end{thmB}
One might then try to enumerate this set. One aim of this paper is to do that for finite simple groups
for which we find that $\text{\rm Idem}(G)$ has in general very few elements (see Corollary~\ref{cor covsimple} and Section 11).  Recall that by functoriality $\text{Aut}(G)$ acts on the Schur multiplier  $H_2(G)$ of $G$.
Let  $\text{Inv}\gensub{H_2(G)}$
denote the set of all subgroups of $H_2(G)$ which are invariant (not necessarily pointwise fixed)
under this action.

%%%%%%%%%%%%%%%%%%%%%%%%%%%%%%%%%%%%%%%%%%%%%%%%%%%%%%%%%%%%%%%%
\begin{thmC}
%%%%%%%%%%%%%%%%%%%%%%%%%%%%%%%%%%%%%%%%%%%%%%%%%%%%%%%%%%%%%%%
Let $G$ be a finite simple group.
There is a bijection between $\text{\rm Idem}(G)$ and the set:
\[\{0\}\coprod \text{\rm Inv}\gensub{H_2(G)}.
\]
\end{thmC}

The bijection in the above theorem can be described explicitly. The value of an idempotent functor on $G$ could be the trivial group. This corresponds to the element $0$ in the above set. If
the value is not trivial, then it can be constructed as follows: first take the universal central
extension $H_{2}(G)\triangleleft E\ra G$, then, for an invariant subgroup $K\subset H_2(G)$,   take the quotient
$E/K$. Such quotients are exactly the non-trivial values  idempotent functors take on  a simple group $G$.

The composition $( \phi'\phi, \epsilon'_{\phi(-)}\epsilon)$ of  two idempotent functors $(\phi,\epsilon)$ and   $(\phi',\epsilon')$ is not, in general, idempotent.
Such compositions give  new operations on groups. We can then study the orbits of
the action of this broader collection of operations:

\begin{Def}\label{def iteratedidem}
Let $n$ be a positive integer.
\[\text{\rm Idem}^n(G):=\{\text{isomorphism classes of }\phi_1\cdots \phi_n(G)\ |\
\text{for all $i$, }(\phi_i,\epsilon_i)\text{ is idempotent}\}\]

Since the identity functor, with the augmentation given  by the identity, is idempotent, $\text{\rm Idem}^1(G)\subset \text{\rm Idem}^2(G)\subset\text{\rm Idem}^3(G)\cdots$. 
Using this increasing sequence of inclusions, we define:
\[\text{\rm Idem}^{\infty}(G):=\bigcup_{k\geq 1}\text{\rm Idem}^{k}(G).\]
\end{Def}

By applying all idempotent functors to a finite group, according to Theorem B, we get only  finitely many values. We can then apply idempotent functors to these newly obtained groups to get
yet again some finite set of groups. We can keep iterating this procedure. It turns out that this process eventually stabilizes,   the set of values remains unchanged, and by  repeating these operation arbitrary number of times  we  get only a finite number of different isomorphism classes of groups (see Corollary~\ref{cor ideminffin} and Proposition~\ref{prop idemiterfinsimple}):

\begin{thmD}
If $G$ is finite, then $\text{\rm Idem}^{\infty}(G)$ is a finite set. If in addition $G$ is simple, then
$\text{\rm Idem}^{\infty}(G)=\text{\rm Idem}^{2}(G)$.
\end{thmD}

%For the convenience of the reader we include here the partition of the
%paper into sections.

%\tableofcontents
%\clearpage

%%%%%%%%%%%%%%%%%%%%%%%%%%%%%%%%%%%%%%%%%
%%%%%%%%%%%%%%%%%%%%%%%%%%%%%%%%%%%%%%%%%%
%%%%%%%%%%%%%%%%%%%%%%%%%%%%%%%%%%%%%%%%%%
\section{Idempotent functors and cellular covers}
%%%%%%%%%%%%%%%%%%%%%%%%%%%%%%%%%%%%%%%%%
%%%%%%%%%%%%%%%%%%%%%%%%%%%%%%%%%%%%%%%%%%%%
%%%%%%%%%%%%%%%%%%%%%%%%%%%%%%%%%%%%%%%%%%%

One well known example of an idempotent functor is  given by the cellularization and many of the present results are
extensions and generalizations of results and technique developed for these functors 
(see for example~\cite{Ca1, CDFS, RoSc,Fl, FGS}.) Recall that,  for any group $A$,
there is a functor $\cell_A\colon\text{Groups}\ra\text{Groups}$
and a natural transformation $c_{A,G}:\cell_AG\to G$.
%This augmentation is the initial object in
%$\text{Groups}|_{G}$ subject to the condition that
This  augmentation is required to fulfil the following properties:
\begin{itemize}
\item
$\text{Hom}(A,c_{A,G}):\text{Hom}(A,\cell_AG)\ra \text{Hom}(A,G)$ is a bijection.

\item
%This means that
For any group homomorphism $f:X\ra Y$ for which $\text{Hom}(A,f)$ is a bijection,
$\Hom(\cell_AG,f)$ is also a bijection.
\end{itemize}

Recall from \cite{FGS} the notion of a {\bf cellular cover}  of a group $G$.
This is a homomorphism $c: A\ra G$ such that
$\text{\rm Hom}(A,c):\text{\rm Hom}(A,A)\ra \text{\rm Hom}(A,G)$
is a bijection.  There are two facts to bare in mind when discussing
cellular covers:
\begin{itemize}
\item
if $c: A\ra G$ is a cellular cover, then $A=\cell_AG$, and

\item
for any idempotent functor of the form $(\cell_A, c_A)$,
$c_{A,G}:\cell_AG\ra G$ is a cellular cover.
\end{itemize}

The purpose of this section is to establish a bijection between
$\text{Idem}(G)$ and equivalence classes of cellular covers of $G$:

%
%
%%%%%%%%%%%%%%%%%%%%%%%%%%%%%%%%%%%%%%%%%%%%%%%%%%%%%%%%%%%%%%%%%
\begin{Def}\label{def cover}
%%%%%%%%%%%%%%%%%%%%%%%%%%%%%%%%%%%%%%%%%%%%%%%%%%%%%%%%%%%%%%%%%%%%
\begin{enumerate}
\item
Two cellular covers $c:A\ra G$ and $d:B\ra G$ are defined  to be equivalent if
there is an isomorphism $h:A\ra B$ for which $dh=c$.
\item
The symbol $\cov{G}$ denotes the collection of equivalence classes of cellular covers of $G$.
\end{enumerate}
\end{Def}

%%%%%%%%%%%%%%%%%%%%%%%%%%%%%%%%%%%%%%%%%%%%%%%%%%%%%%%%%%
%
\begin{lemma}\label{lem isoAB}
%%%%%%%%%%%%%%%%%%%%%%%%%%%%%%%%%%%%%%%%%%%%%%%%%%%%%%%%%%%
\begin{enumerate}
\item
If $c:A\ra G$ and $d:B\ra G$ are cellular covers for which $\Hom(A,d)$
and $\Hom(B, c)$ are bijections, then $c$ and $d$ are equivalent.

\item
If  $c:A\ra G$ and $d:B\ra G$ are cellular covers for which  $A$ and $B$ are isomorphic groups,
 then $c$ and $d$ are equivalent.
\end{enumerate}
\end{lemma}
\begin{proof}
For (1) note that the bijectivity of $\text{Hom}(A,d):\text{Hom}(A,B)\ra \text{Hom}(A,G)$ implies that there is
a unique $h:A\ra B$ for which $dh=c$.  The same argument gives a unique $h':B\ra A$ for which
$ch'=d$.  We thus get equalities $ch'h=c$ and $dhh'=d$ which imply that
$h'h=\text{id}_A$ and $hh'=\text{id}_B$, here we use again that $c$ and $d$ are cellular covers.
Part (2) follows from (1) since if $A$ and $B$ are isomorphic, then the hypothesis of
(1) holds.
\end{proof}

%%%%%%%%%%%%%%%%%%%%%%%%%%%%%%%%%%%%%%%%%%%%%%%%%%%%%%%%%%%%%%%
\begin{prop}\label{prop identifyingidemp}
%%%%%%%%%%%%%%%%%%%%%%%%%%%%%%%%%%%%%%%%%%%%%%%%%%%%%%%%%%%%%%%%
Let $G$ be a group. The function  assigning to the equivalence class of a cellular cover $c: A\ra G$ the group $A$
is a bijection between
$\cov{G}$ and $\text{\rm Idem}(G)$.
\end{prop}

\begin{proof}
Let $c:A\ra G$ be a cellular cover.  Then, by the remarks above
$A$ is isomorphic to $\cell_AG$ and hence it is the value of
the idempotent functor $(\cell_A, c_A)$. Thus our map takes an equivalence class
of a cellular cover $c:A\ra G$ to an element in $\text{\rm Idem}(G)$.
Also, by Lemma \ref{lem isoAB}(2), our map is injective.

It remains to show that all elements in $\text{\rm Idem}(G)$ are obtained in this way.
Let $(\phi,\epsilon)$ be an idempotent functor.  We claim that $\epsilon_G:\phi(G)\ra G$ is a cellular cover,
i.e., the map of sets  $\text{Hom}(\phi(G),\epsilon_G):\text{Hom}(\phi(G),\phi(G))\ra \text{Hom}(\phi(G),G)$
 is a bijection.

 We show  the surjectivity first.
Let   $f:\phi(G)\ra G$ be a homomorphism. Consider the following commutative square:
\[\xymatrix{
\phi^2(G)\dto_{\epsilon_{\phi(G)}}\rto^{\phi(f)} & \phi(G)\dto^{\epsilon_G}\\
\phi(G)\rto^{f} & G
}\]
Since  $\epsilon_{\phi(G)}$ is an isomorphism,  $f$ factors through $\epsilon_G$. As this happens for any $f$,
$\text{Hom}(\phi(G),\epsilon_G)$ is a surjection.

It remains to show the injectivity of  $\text{Hom}(\phi(G),\epsilon_G)$.
Let  $f,g:\phi(G)\ra \phi(G)$ be two homomorphisms for which  $\epsilon_G f=\epsilon_G g$.
Consider the following commutative diagram:

\[\xymatrix{
\phi^2(G)\ar@<1ex>[r]^{\phi(f)}  \ar@<-1ex>[r]_{\phi(g)}\dto_{\epsilon_{\phi(G)}} & \phi^{2}(G)\dto^{\epsilon_{\phi(G)}}\rto^{\phi(\epsilon_G)} &
\phi(G)\dto^{\epsilon_G}\\
\phi(G)\ar@<1ex>[r]^{f}  \ar@<-1ex>[r]_{g} & \phi(G)\rto^{\epsilon_G} & G
}\]
Since $\epsilon_G f=\epsilon_G g$, then $\phi(\epsilon_G)\phi(f)=\phi(\epsilon_G)\phi(g)$. As
$\phi(\epsilon_G)$ is an isomorphism,  $\phi(f)=\phi(g)$. Consequently
$f\epsilon_{\phi(G)}=g\epsilon_{\phi(G)}$. Again since  $\epsilon_{\phi(G)}$ is an isomorphism, $f=g$.
\end{proof}

According to~\ref{prop identifyingidemp}, for any group $A$ representing  an element in $\text{\rm Idem}(G)$, there is a unique, up to an isomorphism of $A$, homomorphism $c:A\ra G$ which is a cellular cover.

%%%%%%%%%%%%%%%%%%%%%%%%%%%%%%%%%%%%%%%%%%%%%%%%%%%%%
%%%%%%%%%%%%%%%%%%%%%%%%%%%%%%%%%%%%%%%%%%%%%%%%%%
%%%%%%%%%%%%%%%%%%%%%%%%%%%%%%%%%%%%%%%%%%%%%%%%%%
\section{Nilpotent groups}
%%%%%%%%%%%%%%%%%%%%%%%%%%%%%%%%%%%%%%%%%%%%%%%%%%%%
%%%%%%%%%%%%%%%%%%%%%%%%%%%%%%%%%%%%%%%%%%%%%%%%%%%
%%%%%%%%%%%%%%%%%%%%%%%%%%%%%%%%%%%%%%%%%%%%%%%%%%%
If the only group homomorphism from  $G$ to  $X$ is the trivial homomorphism, then we write
$\text{Hom}(G,X)=0$. %If $\text{Hom}(G,H)=0$, then we say that $H$�is $G$-null.
The property of not having any non-trivial homomorphism into a given group is not preserved by subgroups in general. For example $\text{Hom}({\mathbf Q},{\mathbf Z}/p)=0$, however
for the subgroup ${\mathbf Z}\subset {\mathbf Q}$,
$\text{Hom}({\mathbf Z},{\mathbf Z}/p)\not=0$.  Dually the property of not receiving any non-trivial homomorphism  from a given group is not preserved by quotients in general. For example $\text{Hom}({\mathbf Z}/p, {\mathbf Q})=0$,
however for the quotient $\mathbf Q\onto {\mathbf Z}/p^{\infty}$, $\text{Hom}({\mathbf Z}/p,{\mathbf Z}/p^{\infty})\not=0$. The reason  is that the subgroup  $\mathbf Z$ and the quotient ${\mathbf Z}/p^{\infty}$ of  $\mathbf Q$
are too ''small''. The aim of this section is to show that if  $G$  is nilpotent and
$\text{Hom}(G,X)=0$, then $\text{Hom}(H,X)=0$ for any ''big'' subgroup $H\subset G$. Dually,
if $\text{Hom}(X,G)=0$, then $\text{Hom}(X,H)=0$ for any ''big'' quotient $G\onto H$.
The adjective big is clarified by the following proposition:

%%%%%%%%%%%%%%%%%%%%%%%%%%%%%%%%%%%%%%%%%%%%%%%%%%%%%%%%%%%%%%%%
\begin{prop}\label{prop bigsubgroupsandquotients}
%%%%%%%%%%%%%%%%%%%%%%%%%%%%%%%%%%%%%%%%%%%%%%%%%%%%%%%%%%%%%%%%%%
Let $G$ be a nilpotent group and $X$ be a group.
\begin{enumerate}
\item If $\text{\rm Hom}(G,X)=0$, then, for any $i$, $\text{\rm Hom}(\Gamma_i(G),X)=0$.
\item If $\text{\rm Hom}(G,X)=0$, then $\text{\rm Hom}(H,X)=0$  for any  normal  subgroup  $H\trianglelefteq G$ for which $G/H$ is finitely generated.
\item If $\text{\rm Hom}(X,G)=0$, then $\text{\rm Hom}(X,G/H)=0$ for any finite normal subgroup $H\trianglelefteq G$.
\end{enumerate}
\end{prop}
This proposition can be used to show that certain groups are finite:

%%%%%%%%%%%%%%%%%%%%%%%%%%%%%%%%%%%%%%%%%%%%%%%%%%%%%%%%%%%%%%%%%%%%%%%
\begin{cor}\label{cor finitetest}
%%%%%%%%%%%%%%%%%%%%%%%%%%%%%%%%%%%%%%%%%%%%%%%%%%%%%%%%%%%%%%%%%%%%%%
Let $K\trianglelefteq H$ be a normal  subgroup  such that:
\begin{itemize}
\item[(a)]
$K$ is nilpotent,
\item [(b)]
$\text{\rm Hom}(H,K)=0$,
\item[(c)] $H/K$ is finitely generated,
\item[(d)] for some $j\geq 1$, $K\cap \Gamma_{j}(H)$ is finite.
\end{itemize}
Then $K$ is finite and  $K\subset \Gamma_{i}(H)$ for any $i\geq 1$.
\end{cor}
\begin{proof}
Note that to prove the corollary it is enough to show
$K=K\cap\Gamma_{i}(H)$ for any $i\geq j$. Let $G:=H/\Gamma_{i}(H)$
and consider $Q:=K/\big(K\cap\Gamma_{i}(H)\big)$.
Then $Q$ is (isomorphic to) a subgroup of $G$ and $G/Q$ is finitely generated.
Assume that $Q$ is non-trivial.  Then, by Lemma \ref{lemma keynilpotent}(2)
below, $\Hom(G,Q)\ne 0$, and then also $\Hom(H,Q)\ne 0$.

On the other hand, by hypothesis (b) $\text{\rm Hom}(H,K)=0$.
As $K\cap \Gamma_{i}(H)$ is finite and $K$ is nilpotent Proposition ~\ref{prop bigsubgroupsandquotients}(3),
implies that $\text{\rm Hom}(H,Q)=0$.  This contradiction completes the proof.
\end{proof}

Our key tool to prove~\ref{prop bigsubgroupsandquotients} is the following basic lemma. It provides existence of certain
non-trivial homomorphisms. These are know facts whose proofs are included for self containment.

%%%%%%%%%%%%%%%%%%%%%%%%%%%%%%%%%%%%%%%%%%%%%%%%%%%%%%%%%%%%%%%
%
\begin{lemma}\label{lemma keynilpotent}
%%%%%%%%%%%%%%%%%%%%%%%%%%%%%%%%%%%%%%%%%%%%%%%%%%%%%%%%%%%%%%%%
Let $G$ be  a nilpotent group.
\begin{enumerate}
\item For any $i$, there is a  non-trivial  homomorphism from $G_{\text{\rm ab}}$ to any non-trivial quotient of $\Gamma_i(G)$.
\item  If $H\trianglelefteq G$ is a   normal  subgroup for which $G/H$ is finitely generated, then
there is a non-trivial homomorphism from $G_{\text{\rm ab}}$ to any non-trivial quotient  of $H$.
\item Let  $H$ be a  finite group. If  $\text{\rm Hom}(G,H)=0$, then the function $G\ni g\mapsto g^{|H|}\in G$
is a surjection.
\item If $H$ is a finite proper normal subgroup of $G$, then there is a non-trivial homomorphism $(G/H)_{\text{\rm ab}}\ra G$.
\end{enumerate}
\end{lemma}
\begin{proof}
Recall  that, for any sequence
of elements $x_{2},\ldots,x_{i}$ in $G$,
the following map of sets is a group homomorphism:
\[G_{\text{\rm ab}}=G/[G,G]\ni x[G,G] \mapsto [x,x_{2},\ldots,x_{i}]\Gamma_{i+1}(G)\in\Gamma_{i}(G)/\Gamma_{i+1}(G). \]
We will refer to it as the homomorphism given by the sequence  $x_{2},\ldots,x_{i}$ and denote it by
$[-,x_2,\ldots,x_i]$.
\smallskip

\noindent (1):\quad
The proof is by reverse induction on $i$.
Let  $j$ be maximal such that $\Gamma_j(G)$ is non-trivial. For any proper subgroup $K\subset \Gamma_j(G)$,
there is an element  $[x_1,\ldots, x_j]$ of $\Gamma_j(G)$ which does not belong to $K$.
The desired non-trivial homomorphism is given by the composition:
\[\xymatrix{G_{\text{ab}}\rrto^(.45){[-,x_2,\ldots,x_j]} & & \Gamma_j(G)\ar@{->>}[rr]^(.43){\text{quotient}} & & \Gamma_j(G)/K.
}\]

Assume $i<j$. Let $K$ be a proper normal subgroup of $\Gamma_i(G)$. There are two possibilities:

\begin{list}{\labelitemi}{\leftmargin=1.5em \setlength\labelwidth{1in}}
\item
$\Gamma_{i+1}(G)\subset K$:\quad  Let $[x_1,\ldots,x_i]$ be an element of $\Gamma_i(G)$ which does not belong to $K$. The desired non-trivial homomorphism is given by the composition:
\[\xymatrix{G_{\text{ab}}\rrto^(.35){[-,x_2,\ldots,x_i]} & & \Gamma_i(G)/\Gamma_{i+1}(G)\ar@{->>}[rr]^(.53){\text{quotient}} & & \Gamma_i(G)/K.
}\]

\item $\Gamma_{i+1}(G)\not\subset K$:\quad  Let $f:G_{\text{\rm ab}}\ra \Gamma_{i+1}(G)/\big(\Gamma_{i+1}(G)\cap K\big)$ be a non-trivial homomorphism which exists by the inductive assumption.
The desired non-trivial homomorphism is given by the composition:
\[\xymatrix{G_{\text{\rm ab}}\rto^(.25){f} & \Gamma_{i+1}(G)/\big(\Gamma_{i+1}(G)\cap K\big)\ar@{^(->}[r] &
 \Gamma_{i}(G)/K.
}\]
\end{list}
%The composition of the homomorphism $[-,x_2,\ldots,x_j]:G_{\text{\rm ab}}\ra \Gamma_j(G)$ with the quotient
%$\Gamma_j(G)\ra \Gamma_j(G)/K$ is the desired non-trivial homomorphism $G_{\text{\rm ab}}\ra  \Gamma_j(G)/K$.
%We proceed by reverse induction on $i$. Assume $i<j$. Let $K$ be a proper normal subgroup of $\Gamma_i(G)$.
%There are two possibilities: either $\Gamma_{i-1}(G)$ is a subgroup of $K$ or not.
%In the first case let $[x_1,\ldots,x_i]$ be an element of $\Gamma_i(G)$ which does not belong to $K$.
%The composition of $[-,x_2,\ldots,x_i]:G_{\text{\rm ab}}\ra \Gamma_i(G)/\Gamma_{i-1}(G)$ and the quotient
%$\Gamma_i(G)/\Gamma_{i-1}(G)\ra \Gamma_i(G)/K$ is the desired non-trivial homomorphism
%$G_{\text{\rm ab}}\ra  \Gamma_j(G)/K$.
%In the case $\Gamma_{i-1}(G)$ is not a subgroup of $K$, the composition of any non-trivial homomorphism
%$G_{\text{\rm ab}}\ra \Gamma_{i-1}(G)/\big(\Gamma_{i-1}(G)\cap K\big)$ (which exists by the inductive assumption)
%with the inclusion $\Gamma_{i-1}(G)/\big(\Gamma_{i-1}(G)\cap K\big)\subset \Gamma_{i}(G)/K$
%is the desired non-trivial homomorphism $G_{\text{\rm ab}}\ra  \Gamma_j(G)/K$.
\smallskip

\noindent (2):\quad
Let $K$ be a proper normal subgroup of $H$. Our aim is to construct a non-trivial homomorphism $G_{\text{ab}}\ra H/K$.
We  consider two cases:
\medskip

\noindent
       {\bf $G/H$ is infinite:}\quad  In this case, for some $i$,
       $\Gamma_{i}(G/H)/\Gamma_{i+1}(G/H)$ is a finitely generated infinite abelian group. For this $i$,
	the group of integers ${\mathbf Z}$ is a quotient of $\Gamma_{i}(G/H)/\Gamma_{i+1}(G/H)$ and so $	
	{\mathbf Z}$ is a quotient  $\Gamma_{i}(G)$.  By (1),  there
	is a non-trivial homomorphism $G_{ab}\ra {\mathbf Z}$. Its image must be  isomorphic to ${\mathbf Z}$ and
	consequently there is a surjection $G_{\text{ab}}\onto {\mathbf Z}$. The composition of this surjection  with 	 any non-trivial homomorphism  $\mathbf Z\ra H/K$ is the desired non-trivial homomorphism $G_{\text{ab}}\ra 	H/K$.
\medskip

\noindent
{\bf $G/H$ is finite:}\quad  The desired non-trivial homomorphism $G_{\text{ab}}\ra H/K$  will be
	 constructed by induction on the 		order of $G/H$.
	\begin{list}{\labelitemi}{\leftmargin=1.5em \setlength\labelwidth{1in}}
	\item $|G/H|=1$:\quad  Since $H=\Gamma_0(G)=G$, the existence of the non-trivial homomorphism $G_			 {\text{ab}}\ra \Gamma_0(G)/K$ is given by statement (1).
	\item $|G/H|$ is a prime number $p$:\quad In this case  $G/H$ is isomorphic to
		$ \mathbf{Z}/p$.
		We proceed by  induction on the nilpotency class of $G$.

		 \begin{list}{\labelitemi}{\leftmargin=1.5em}
		 \item[--] $G$ is abelian:\quad  The multiplication by $p$ homomorphism
			$p:G/K\ra G/K$ factors as:
			\[\xymatrix{
			G/K\rto^{f}\ar@/ _ 15pt/[rr]|{p} & H/K\rmono & G/K.
			}\]
			There are two possibilities:
			\begin{list}{\labelitemi}{\leftmargin=1.5em}
			 \item[$\ast$]   $p:G/K\ra G/K$ is non-trivial:\quad  In this case $f$ can not be trivial either and
				the desired homomorphism can be taken to be the composition:
				 \[\xymatrix{
				G\ar@{->>}[rr]^(.4){\text{quotient}}& & G/K\rto^f & H/K.
				}\]
			 \item[$\ast$]   $p:G/K\ra G/K$ is trivial:\quad
				The abelian  groups  $G/K$ and $H/K$ are then  $ \mathbf{Z}/p$-vector spaces.
				As   $H/K$ is not trivial, it contains $ \mathbf{Z}/p$ as a subgroup and
				the desired homomorphism can be taken to be the composition:
				\[\xymatrix{
				G\ar@{->>}[rr]^(.35){\text{quotient}}& & G/H=\mathbf{Z}/p\ar@{^(->}[r]  & H/K.
				}\]
			\end{list}
		 \item[--]
			Let $i=\text{max}\{j\ |\ \Gamma_{j}(G)\not=1\}>0$:\quad   Since $G/H$ is abelian,
			$\Gamma_{i}(G)	$ is a subgroup of $H$. There are two possibilities:
			 \begin{list}{\labelitemi}{\leftmargin=1.5em}
			 \item[$\ast$]
				 $\Gamma_{i}(G)\subset K$:\quad
				Consider the following sequence of groups:
				\[K/\Gamma_{i}(G)\ \trianglelefteq\  H/\Gamma_{i}(G)\ \trianglelefteq\  G/\Gamma_{i}(G).\]
				As the nilpotence class of $G/\Gamma_{i}(G)$ is smaller than that of $G$, by the inductive 			
				assumption there is a non-trivial homomorphism:
				\[f:(G/\Gamma_{i}(G))_{\text{ab}}\ra
				\big(H/\Gamma_{i}(G)\big)/\big(K/\Gamma_{i}(G)\big)=H/K\]
				The following composition gives the desired non-trivial  homomorphism:
				\[\xymatrix{G_{\text{ab}}\ar@{->>}[rr]^(.36){\text{quotient}} &&
				(G/\Gamma_{i}(G))_{\text{ab}}\rto^(.6){f}& H/K.}\]
			 \item[$\ast$]
				 $\Gamma_{i}(G)\not\subset K$:\quad
				Let $[x_1,\ldots,x_i]$ be element of $\Gamma_i(G)$
				which does not belong to $K$. The following
				composition gives the desired non-trivial  homomorphism:
				\[\xymatrix{
				G_{\text{ab}}\rrto^(0.45){[-,x_2,\ldots, x_i]} & &
				\Gamma_{i}(G)\ar@{^(->}[r] & H\ar@{->>}[rr]^(.39){\text{quotient}} && H/K.
				}\]
			\end{list}
		\end{list}
	\item $|G/H|>1$ and $|G/H|$ is not a prime:\quad  Since $G/H$ is finite and nilpotent and its
		 order is not a prime number, there is a sequence of
		proper normal subgroups  $H\triangleleft L \triangleleft G$.
		By the inductive 	
		assumption there is a non-trivial homomorphism $f:L_{\text{ab}}\ra H/K$.
		Let $K_f$ be the kernel of $f$. By the inductive assumption applied to $L \trianglelefteq G$
		there is also a non-trivial
		homomorphism $g:G_{\text{ab}}\ra L_{\text{ab}}/K_f$. The following
		composition gives the desired non-trivial  homomorphism:
		\[\xymatrix{
		G_{\text{ab}}\rto^(.45){g} &  L_{\text{ab}}/K_f\ar@{^(->}[r] &H/K.
		}\]
	\end{list}
\smallskip

\noindent (3):\quad
We prove the statement by induction on the nilpotence class of $G$.
\begin{list}{\labelitemi}{\leftmargin=1.5em \setlength\labelwidth{1in}}
\item $G$ is abelian:\quad
If $|H|=1$, the statement is clear. Assume $|H|>1$ and let $p$ be a prime dividing $|H|$.
Since  the group ${\mathbf Z}/p$ is a subgroup of $H$, we also have
$\text{\rm Hom}(G,{\mathbf Z}/p)=0$. This means that $G\otimes {\mathbf Z}/p=0$ and hence the multiplication by $p$ homomorphism
$p:G\ra G$ is a surjection. As this happens for all the primes dividing $|H|$, same is true
for the homomorphism $G\ni g\mapsto |H|g\in G$.
\item  Let $i=\text{max}\{j\ |\ \Gamma_{j}(G)\not=1\}>0$:\quad
We claim that:
\[\text{Hom}(G/\Gamma_{i}(G),H)=0,\ \ \ \ \ \ \text{Hom}(\Gamma_{i}(G),H)=0.\]
The first equality is clear as $G/\Gamma_{i}(G)$ is a quotient of $G$ and $G$ has no non-trivial homomorphisms into $H$.
Let $f:\Gamma_{i}(G)\ra H$ be a homomorphism  and $L\subset H$ be its image.
If $L$ were non-trivial, then by statement (1), there would be a non-trivial
homomorphism $g:G_{\text{ab}}\ra L$.
The following composition would be then a non-trivial homomorphism from $G$ to $H$ which contradicts our assumption:
	   \[\xymatrix{
	G\ar@{->>}[rr]^(0.44){\text{quotient}} && G_{\text{ab}}\rto^{g} & L\ar@{^(->}[r] & H.
	}\]

Let $g\in G$. We need to show that there is an element whose $|H|$-th power is $g$. Since the nilpotence class of $G/\Gamma_{i}(G)$
is smaller than that of $G$, by the inductive assumption, there is $h\in G$ such that, for some $a\in \Gamma_{i}(G)$, $h^{|H|}a=g$. As  $\Gamma_{i}(G)$ is abelian and $ \text{Hom}(\Gamma_{i}(G),H)=0$, there is also $b\in \Gamma_{i}(G)$ for which  $b^{|H|}=a$.
The triviality of $\Gamma_{i+1}(G)$ implies that $b$ is central in $G$. It follows that $g=(hb)^{|H|}$.

\end{list}
\smallskip

\noindent (4):\quad
We first claim that we may assume that $G/H$ is abelian.  Indeed suppose
part (4) holds in this case.  Consider $C_G(H)$.  If $C_G(H)$ is
not central in $G$, then pick $x\in Z_2(G)\cap C_G(H)$.
The map $g\mapsto [g,x],$ for all $g\in G,$ is
a non-trivial homomorphism from $G$ to $Z(G),$ that contains $H$ in its kernel.
Hence it induces a non-trivial homomorphism $G/H\to G$ and we are done.

Hence $C_G(H)=Z(G)$  and since $G/C_G(H)$ is
isomorphic to a subgroup of $\Aut(H)$ and $H$ is finite,
we conclude that $G/Z(G)$ is finite and hence, by a theorem of Schur, $[G,G]$
is finite.  But now, by our hypothesis there exists
a non-trivial homomorphism $G/H[G,G]\ra G$ and hence
also a non-trivial homomorphism $G/H\to G$.

It remains to prove part (4) under the hypothesis that $G/H$ is abelian, i.e.,
$[G,G]$ is a subgroup of $H$.  Under this assumption we proceed by induction on the order $|H|$
to  show the existence of a non-trivial homomorphism $G/H\ra G$.

If  $|H|=1$,  then $G$ and $G/H$ are isomorphic and the statement is  clear.

Assume $|H|>1$.
If there is a non-trivial homomorphism $G/H\ra H$, then its composition
with the inclusion $H\subset G$ gives a  non-trivial homomorphism
$G/H\ra G$ and the statement is proven. Thus we can assume $\text{Hom}(G/H,H)=0$ and consequently, according to (3), $G/H$ is $|H|$-divisible.

Consider the map
\[
h: G/H\to G/[G,G]\text{ defined by }Hg\mapsto [G,G]g^{|H|}.
\]
It is easy to check that this is a group homomorphism, its kernel
$K_h$ is annihilated by $|H|$ and since $G/H$ is $|H|$-divisible,
$\text{Hom}(G/H,K_h)=0$.

We can use  $h$ to form the following pull-back square:
\[\xymatrix{
& K_h\ar @{=}[rr]\dmono & &K_h\dmono\\
[G,G]\ar @{=}[d] \rmono & P\dto^{h'}\ar@{->>}[rr] && G/H\dto^{h}\\
[G,G]\rmono & G\ar@{->>}[rr]^(.4){\text{quotient}} && G/[G,G]
}\]
There are two possibilities:
\begin{list}{\labelitemi}{\leftmargin=1.5em \setlength\labelwidth{1in}}
\item $[G,G]=H$:\quad
In this case the non-trivial
homomorphism $G/H=G_{\text{ab}}\ra G$ is given by statement (1).
\item $[G,G]\subsetneq H$:\quad   Notice that
the image $P \ra G/H$ is abelian,
so we can apply  the inductive assumption (with $P$ in place of $G$ and $[G,G]$ in place of $H$)
to deduce that  there is
a non-trivial homomorphism $\alpha:G/H\ra P$.  Consider the composition
of $\alpha$ with the vertical homomorphism $h':P\ra G$ in the above
diagram.   If this composition were trivial then $\alpha:G/H\ra P$ would factor through
$K_h\subset P$. This however is impossible since there are no non-trivial homomorphisms from $G/H$ to $K_h$.
\qedhere
\end{list}
\end{proof}

We are now ready to prove:
\begin{proof}[Proof of~\ref{prop bigsubgroupsandquotients}]
\noindent (1)\&(2):\quad Let $H$ be either $\Gamma_i(G)$ or a normal
subgroup of $G$ for which $G/H$ is finitely generated. Assume $\text{Hom}(G,X)=0$.
Let $f:H\ra X$ be a homomorphism. If $f$ is non-trivial,
then according by Lemma~\ref{lemma keynilpotent}(1\&2), there is a non-trivial
homomorphism $G_{\text{ab}}\ra \text{im}(f)$. This implies
the existence of a non-trivial homomorphism $G\ra X$ contradicting
$\text{Hom}(G,X)=0$. Hence $\text{Hom}(H,X)=0$.
\smallskip

\noindent (3):\quad
Let
$f:X\ra G/H$ be an arbitrary homomorphism. Consider   its image $B\subset G/H$
and the following pull-back square:
\[\xymatrix{ H\rmono \ar @{=}[d] & P\dmono\ar@{->>}[rr] & & B\dmono\\
H\rmono & G\ar@{->>}[rr]^(.4){\text{quotient}} & &G/H
}\]
According to Lemma~\ref{lemma keynilpotent}(4), if $H\subset P$ were a proper subgroup, then there would be a non-trivial homomorphism
$B\ra P$. The composition of this homomorphism with the injection
$P\hookrightarrow G$ would give a non-trivial homomorphism $B\ra G$.
This however is impossible since the composition of this homomorphism with the $f:X\onto B$ would be also non-trivial
contradicting the assumption $\text{Hom}(X,G)=0$. We can conclude that $H=P$ and hence $B=0$. Consequently $f$ is the trivial homomorphism and $\text{Hom}(X,G/H)=0$.
\end{proof}

%%%%%%%%%%%%%%%%%%%%%%%%%%%%%%%%%%%%%%%%%%%%%%%%%%%%
%%%%%%%%%%%%%%%%%%%%%%%%%%%%%%%%%%%%%%%%%%%%%%%%%%
%%%%%%%%%%%%%%%%%%%%%%%%%%%%%%%%%%%%%%%%%%%%%%%%%%%
%section4
\section{Generalized subgroups}
%%%%%%%%%%%%%%%%%%%%%%%%%%%%%%%%%%%%%%%%%%%%%%%%%%%%
%%%%%%%%%%%%%%%%%%%%%%%%%%%%%%%%%%%%%%%%%%%%%%%%%%%%%%
%%%%%%%%%%%%%%%%%%%%%%%%%%%%%%%%%%%%%%%%%%%%%%%%%%%%%%
Ultimately we would like to classify elements of $\text{Idem}(G)$  for a finite group $G$ using some classical invariants.
According to~\ref{prop identifyingidemp} this is equivalent to the enumeration of
  $\cov{G}$.   Unfortunately we  are unable to enumerate $\cov{G}$.
It turns out however that it is easier to give a classification  for a bigger collection which is the subject of
Section~\ref{sec gensuboffinite}. In this section we define this bigger collection
which we call {\bf generalized subgroups,} and discuss
some properties of its elements.

%%%%%%%%%%%%%%%%%%%%%%%%%%%%%%%%%%%%%%%%%%%%%%%%%%%%%%%%%%%%%%%%%%%%%%%%%%%%
\begin{Def}\label{def gensubgroups}
%%%%%%%%%%%%%%%%%%%%%%%%%%%%%%%%%%%%%%%%%%%%%%%%%%%%%%%%%%%%%%%%%%%%%%%%%
Let $G$ be a group.
\begin{enumerate}
\item A homomorphism $a:X\ra G$ is called a {\bf generalized subgroup} of $G$ if
$\text{Hom}(X,a):\text{Hom}(X,X)\ra \text{Hom}(X,G)$ is an injection   of sets
(but not necessarily a bijection as it is in the case of a cellular  cover of $G$).
\item
Two generalized subgroups $a:X\ra G$ and $b:Y\ra G$ are defined  to be equivalent if there is an isomorphism $h:X\ra Y$ for which $bh=a$.
\item
The symbol $\gensub{G}$ denotes the collection of equivalence classes of generalized subgroups of $G$.
\end{enumerate}
\end{Def}

We start the study of   generalized subgroups of $G$ by giving their direct characterization:

%%%%%%%%%%%%%%%%%%%%%%%%%%%%%%%%%%%%%%%%%%%%%%%%%%%%%%%%%%%%%%%%%%%%%%%%%
\begin{prop}\label{prop charcCmono}
%%%%%%%%%%%%%%%%%%%%%%%%%%%%%%%%%%%%%%%%%%%%%%%%%%%%%%%%%%%%%%%%%%%%%
A homomorphism
$a:X\ra G$ is a generalized  subgroup of $G$ if and only if the following conditions are satisfied:
\begin{itemize}
\item[(a)]$\text{\rm Ker}(a)$ is a central subgroup of $X$,
\item[(b)] $\text{\rm Hom}(X,\text{\rm Ker}(a))=0$.
\end{itemize}
\end{prop}
\begin{proof}
Assume first that $a:X\ra G$ is a generalized  subgroup of $G$.
Let $x\in \text{\rm Ker}(a)$.
Consider the identity $\text{id}:X\ra X$ and the conjugation $c_x:X\ra X$.
Then $a c_x=a \text{id}$.
It follows that $c_x=\text{id}$ and hence $x$ is  in the center of $X$.  This shows (a).

Consider now the trivial homomorphism
$X\ra X$ and the composition of some $f:X\ra \text{\rm Ker}(a)$ with the inclusion
$\text{\rm Ker}(a)\subset X$. The compositions of these homomorphisms with $a$
are equal to the trivial homomorphism. Thus  any such $f$ must be trivial and
consequently $\text{\rm Hom}(X,\text{\rm Ker}(a))=0$ which is requirement (b).

Assume that conditions (a) and  (b) are satisfied. We need to show  injectivity of
$\text{Hom}(X,a):\text{Hom}(X,X)\ra \text{Hom}(X,G)$.   Let  $f, g:X\ra X$ be homomorphisms. Assume
$af=ag$. This means that, for any $x\in X$, $f(x)g(x)^{-1}$ belongs to $\text{\rm Ker}(a)$.
We claim that the function $X\ni x\mapsto f(x)g(x)^{-1}\in \text{\rm Ker}(a)$ is a group
homomorphism.  This follows from
the fact that $\text{\rm Ker}(a)$ is central in $X$:
\[f(xy)g(xy)^{-1}=f(x)f(y)g(y)^{-1}g(x)^{-1}=f(x)g(x)^{-1}f(y)g(y)^{-1}.\]
Since  $\text{\rm Hom}(X,\text{\rm Ker}(a))=0$, we can conclude that  $f(x)g(x)^{-1}$ is the identity element for any $x\in X$.  Consequently $f=g$.
\end{proof}

We can use this direct characterization to show  that
generalized subgroups of $G$ inherit certain properties of $G$.

%%%%%%%%%%%%%%%%%%%%%%%%%%%%%%%%%%%%%%%%%%%%%%%%%%%%%%%%%%%%%%%%%%%%%%%%%%
\begin{prop}\label{prop presofsolnilfin}
%%%%%%%%%%%%%%%%%%%%%%%%%%%%%%%%%%%%%%%%%%%%%%%%%%%%%%%%%%%%%%%%%%%%%%%%
Let  $a:X\ra G$ be a generalized subgroup.
\begin{enumerate}
\item If $G$ is nilpotent, respectively solvable, then so is $X$.
\item If $G$ is finite, then so is $X$. Moreover
$\text{\rm Ker}(a)\subset \Gamma_i(X)$ for any $i$.
\item If $G$ is finitely generated and nilpotent, then  $a:X\ra G$
is an injection. In particular $X$ is also finitely generated.
\end{enumerate}
\end{prop}

\begin{proof}
\noindent (1):\quad
Assume $G$  is  nilpotent and $\Gamma_i(G)=0$. We claim that $\Gamma_i(X)=0$. The assumption $\Gamma_i(G)=0$ implies that
$\Gamma_i(X)$ is in the kernel of $a$ and hence, according to~\ref{prop charcCmono}, it is central in $X$.
It follows that $\Gamma_{i+1}(X)=0$ and $X$ is a nilpotent group.
We can now  use Lemma~\ref{lemma keynilpotent}(1). If $\Gamma_i(X)$ were non-trivial,
 there would be a non-trivial homomorphism $X\ra \Gamma_i(X)$. The composition of this homomorphism with the inclusion $\Gamma_i(X)\subset \text{Ker}(a)$ would be then  also non-trivial. This  contradicts the fact that $\text{\rm Hom}(X,\text{\rm Ker}(a))=0$ (see~\ref{prop charcCmono}).
Consequently $\Gamma_i(X)=0$.

Similar argument works for solvable groups. If $G^{(i)}=0$, then $X^{(i)}\subset \text{Ker}(a)$ and
hence $X^{(i)}$ is central in $X$.  \smallskip This implies that $X^{(i+1)}=0$ and consequently $X$ is solvable.

\noindent (2):\quad
Assume $G$ is finite.  We  apply~\ref{cor finitetest} to the subgroup $K_a:=\text{Ker}(a)\trianglelefteq  X$ to  prove that $K_a$ is finite. It would then follow that $X$ is also finite.  Since $K_a$ is central in $X$, it is abelian and hence nilpotent.
This is hypothesis (a) of~\ref{cor finitetest}. Hypothesis (b)
of~\ref{cor finitetest} is condition (b) in~\ref{prop charcCmono}.
As $G$ is finite, then so is its subgroup $a(X)\cong X/K_a$. In particular this quotient is finitely generated and we get hypothesis (c) of~\ref{cor finitetest}. As $X/K_a$ is finite and $K_a$ is central in $X$, the quotient $X/Z(X)$ is also finite.
It follows that the commutator $[X,X]$ is  finite (see~\cite[10.1.4, p.~287]{Rob}). In particular $K_a\cap [X,X]$ is finite and we get
hypothesis (d)  of~\ref{cor finitetest}.  We can then conclude that $K_a$ is a finite group and $K_a\subset \Gamma_i(X)$ for any $i$.
\smallskip

\noindent (3):\quad
Assume $G$ is finitely generated and nilpotent.
As in (2) we will apply~\ref{cor finitetest} to
the subgroup $K_a=\text{Ker}(a)\trianglelefteq  X$.
Hypotheses (a) and (b) of~\ref{cor finitetest} are clear. Since $G$ is finitely
generated and nilpotent, then so is any of its subgroups. In particular  $X/K_a$ is finitely generated. This shows
that hypothesis (c)  of~\ref{cor finitetest} holds.
As $G$ is nilpotent, there is  $i$ for which $\Gamma_i(G)=0$. It then follows that
$\Gamma_i(X)$ is also trivial (see the proof of part (1)). In particular $K_a\cap \Gamma_i(X)$ is finite.
We can conclude that $K_a\subset \Gamma_i(X)=0$ and hence $a$ is an injection.
\end{proof}

%%%%%%%%%%%%%%%%%%%%%%%%%%%%%%%%%%%%%%%%%%%%%%%%%%%%%%%%%%%%%%%%%%%%
\begin{cor}\label{cor idemfinfin}
%%%%%%%%%%%%%%%%%%%%%%%%%%%%%%%%%%%%%%%%%%%%%%%%%%%%%%%%%%%%%%%%%%%%%%%%%%%%%%%%%
Let  $(\phi:\text{\rm Groups}\ra\text{\rm Groups},\epsilon: \phi\ra\text{\rm id})$ be an
idempotent functor.    If $G$ is  $s$-nilpotent, or solvable, or finite, or finitely generated and $s$-nilpotent, then so is
 $\phi(G)$.
 \end{cor}
%notechange
\begin{proof}
Recall from~\ref{prop identifyingidemp} that the map $\epsilon_G: \phi(G)\ra G$ is a cellular cover.  In particular, it is
a generalized subgroup, and so the corollary follows from~\ref{prop presofsolnilfin}.
\end{proof}

An inclusion $X\subset G$ is of course  an   example of a generalized subgroup of $G$.
In the case $G$ is finitely generated and nilpotent all the generalized subgroups are inclusions
(see~\ref{prop presofsolnilfin}(3)). In this case
$\gensub{G}$
is simply the collection of all the subgroups of $G$. For example
the set  $\gensub{{\mathbf Z}/n}$, of  subgroups of the cyclic group ${\mathbf Z}/n$ ($n>0$), can be enumerated by the set
$\{k\in{\mathbf Z}\ |\ k>0\text{ and } k\text{ divides }n\}$ of all positive divisors of $n$.
For any such $k$, the corresponding subgroup is generated by $n/k$ and is isomorphic to ${\mathbf Z}/k$.
Note that the inclusion $\mathbf Z/k\subset \mathbf Z/n$ is not only a generalized subgroup but it is also  a cellular  cover. Thus  in this case we have an equality $\cov{ \mathbf Z/n}=\gensub{{\mathbf Z}/n}$. More generally let $A$ be a finite abelian group. Recall that, for an integer $k$, the $k$-torsion subgroup of $A$ consists of all
  $a\in A$ for which   $ka=0$.

%%%%%%%%%%%%%%%%%%%%%%%%%%%%%%%%%%%%%%%%%%%%%%%%%%%%%%%%%%
\begin{prop}\label{prop covfabelian}
%%%%%%%%%%%%%%%%%%%%%%%%%%%%%%%%%%%%%%%%%%%%%%%%%%%%%%%%%%%
If  $A$ is a finite abelian group, then:
\[\cov{A}=\{X\in\gensub{A}\ |\ X \text{ is the  $k$-torsion subgroup of $A$ for some $k$}\}\]
\end{prop}
\begin{proof}
If $X\subset A$ is the $k$-torsion subgroup, then $X$ is $k$-torsion. Since any homomorphism $f:X\ra A$
takes the $k$-torsion elements to the $k$-torsion elements, the image of $f$ sits in the subgroup $X\subset A$.
This means that $X\subset A$ is a cellular cover.

Let $X\subset A$ be a cellular cover and  $k$ be the exponent of $X$, i.e., the smallest positive integer $k$ for which $kX=0$. Since $X$ is finite,  there is a surjection
$X\onto{\mathbf Z}/k$. For any $k$-torsion element $x\in A$, consider the composition of this
surjection  $X\onto {\mathbf Z}/k$ and a homomorphism ${\mathbf Z}/k\ra A$ that maps some generator to the element $x$.
Since $X\subset A$ is a cellular cover,  the image of this composition has to lie in $X$. It follows that $X$ contains all the
$k$-torsion elements of $A$. As $X$ consists of $k$-torsion elements, $X$ is the $k$-torsion subgroup of $A$.
\end{proof}

%%%%%%%%%%%%%%%%%%%%%%%%%%%%%%%%%%%%%%%%%%%%%%%%%%%%%%%%%
%%%%%%%%%%%%%%%%%%%%%%%%%%%%%%%%%%%%%%%%%%%%%%%%%%%%%%%%%%
%%%%%%%%%%%%%%%%%%%%%%%%%%%%%%%%%%%%%%%%%%%%%%%%%%%%%%%%%
%section5
\section{The initial cellular cover}\label{sec initialcover}
%%%%%%%%%%%%%%%%%%%%%%%%%%%%%%%%%%%%%%%%%%%%%%%%%%%%%%%%%
%%%%%%%%%%%%%%%%%%%%%%%%%%%%%%%%%%%%%%%%%%%%%%%%%%%%%%%%
%%%%%%%%%%%%%%%%%%%%%%%%%%%%%%%%%%%%%%%%%%%%%%%%%%%%%%%
The aim of this section is to construct
an  example of a cellular  cover of a finite group, which we call the {\bf initial cellular cover}, a generalization of the  well known universal cover of a perfect group.  This cellular cover will be used in our classification results in the following sections. The information about $G$  needed
 for our construction is  contained in the first two homology groups of $G$.
 We therefore start with a brief recollection of
some facts about  the first two homology groups of finite groups and central extensions.
We do it  for self containment and to set up notation. We refer the reader to e.g.~\cite{G, Ka, Rob, W}, for further information.

Recall that, for two finite abelian groups $A$ and $B$, the groups $A\otimes B$,
$\text{Hom}(A,B)$, $\text{Hom}(B,A)$, $B\otimes A$, $\text{Ext}^1(A,B)$, and $\text{Ext}^1(B,A)$ are isomorphic.
Thus all these groups are zero if and only if the orders of $A$ and $B$ are relatively prime.

For a group $G$,  $H_n(G)$ denotes the $n$-th integral homology group of $G$. The first  homology group $H_1(G)$ is naturally isomorphic to the abelianization $G/[G,G]$ of $G$. Via this isomorphism, for a homomorphism $f:X\ra G$, $H_1(f):X/[X,X]\ra G/[G,G]$ is given by $x[X,X]\mapsto f(x)[G,G]$.
The second homology group $H_2(G)$ is also called the Schur multiplier of $G$.
Recall that if $G$ is finite, then, for any $n$, $H_n(G)$  is also a finite  group
whose exponent  divides
$|G|$.   If $K$ is  finite and  cyclic, then $H_2(K)=0$.

For an abelian group $K$ and a group $G$, a central extension of $G$ by $K$ is a group $X$ containing $K$ in its center
and a surjective homomorphism $f:X\onto G$ for which $\text{Ker}(f)=K$. Two such central extensions
$f:X\onto G$ and $f':X'\onto G$ are equivalent if there is a homomorphism $h:X\ra X'$
for which $f'h = f $ and $h$ restricted to $K$ is the identity. Such $h$ necessarily has to be an isomorphism.
Let us recall that  the equivalence classes of central extensions of $G$ by $K$ form a set which can be identified with the second cohomology group $H^2(G,K)$
(see \cite[11.1.4, p.~318]{Rob}).  An effective tool to study the group $H^2(G,K)$ is the universal
coefficient exact sequence (\cite[11.4.18, p.349]{Rob}):
\[
0\xrightarrow{} {\rm Ext}^1(H_1(G), K)\xrightarrow{} H^2(G,K)\xrightarrow{\mu} \Hom(H_2(G),K)\xrightarrow{}
0
\]

If $f:X\onto G$ represents an  equivalence class of a central extension of $G$ by $K$, then  the homomorphism $\mu(f):H_2(G)\ra K$ is called the differential of $f$.  This differential fits into the following exact sequence (\cite[2.5.6]{Ka}), called the exact sequence of $f$:
\[H_2(X)\xrightarrow{H_2(f)} H_2(G)\xrightarrow{\mu(f)} K\xrightarrow{\alpha} H_1(X)\xrightarrow{H_1(f)} H_1(G)\rightarrow 0
\]
where the homomorphism $\alpha$ is given by $K\ni x\mapsto x[X,X]\in X/[X,X]=H_1(X)$.
This sequence is functorial. This means that, for two central extensions $f:X\onto G$ of $G$ by $K_f$ and $g:Y\onto H$ of $H$ by $K_g$ that fit into the following commutative diagram:
\[\xymatrix{
K_f\ar@{^(->}[r]\dto^{h_1} & X\ar@{->>}[r]^f\dto^{h_1} & G\dto^{h}\\
K_g\ar@{^(->}[r]& Y\ar@{->>}[r]^g & H
}\] the following diagram of the their  exact sequences also commutes:
\[\xymatrix{
H_2(X)\rto^{H_2(f)}\dto_{H_2(h_1)} & H_2(G)\dto|{H_2(h)} \rto^{\mu(f)} & K_f\rto\dto^{h_1} & H_1(X)\rto^{H_1(f)}\dto|{H_1(h_1)} & H_1(G)\dto^{H_1(h)}\rto & 0\\
H_2(Y)\rto^{H_2(g)} & H_2(H) \rto^{\mu(g)} & K_g\rto\ & H_1(Y)\rto^{H_1(g)} & H_1(H)\rto & 0
}\]

Under the assumption that $X$ is finite (it is actually enough to assume that only $K$ is finite), the exact sequence of $f$ can be extended one step further  to an exact sequence, called  also the  exact sequence of $f$:
\[H_1(X)\otimes K\xrightarrow{} H_2(X)\xrightarrow{H_2(f)} H_2(G)\xrightarrow{\mu(f)} K\xrightarrow{\alpha} H_1(X)\xrightarrow{H_1(f)} H_1(G)\rightarrow 0
\]

\begin{Def}\label{def 2minus1}
For a finite group  $G$,  $H_{2\setminus 1}(G)$ denotes  the localization
$H_2(G)[S^{-1}]$ where $S$ is the set of primes that divide the order of $H_1(G)$.
\end{Def}

The group  $H_{2\setminus 1}(G)$ is simply the quotient of
$H_2(G)$ by the $S$-torsion, and the localization homomorphism $H_2(G)\ra H_2(G)[S^{-1}]=H_{2\setminus 1}(G)$ is the quotient homomorphism.
Since the orders of  $H_2(G)[S^{-1}]$ and $H_1(G)$ are coprime, the group
$\text{Ext}^1(H_1(G), H_{2\setminus 1}(G))$
is trivial. The  homomorphism $\mu:H^{2}(G,H_{2\setminus 1}(G))\ra \text{Hom}(H_2(G),H_{2\setminus 1}(G))$
 is therefore an isomorphism. It follows that   there is a unique central extension $e_G:E\onto G$ of
$G$ by $H_{2\setminus 1}(G)$ whose differential $\mu(e_G)$ is the localization homomorphism:
\[\xymatrix{H_2(G)\rrto^(.41){\text{localization}}\ar@/ _ 15pt/[rrr]|{\mu(e_G)} & &  H_2(G)[S^{-1}]\ar@{=}[r]&H_{2\setminus 1}(G).
}\]
We call the extension    $e_G:E\onto G$  the {\bf  initial extension} of $G$.
In the case $G$ is perfect, i.e., if $H_1(G)=0$, then $H_{2\setminus 1}(G)=H_{2}(G)$ and the initial
extension is the universal central extension of $G$.

The key property of the initial
extension of a finite group $G$  is that its differential $\mu(e_G):H_2(G)\ra H_{2\setminus 1}(G)$ is a surjection
(this means that $e_G$  is a, so called, stem extension).

%%%%%%%%%%%%%%%%%%%%%%%%%%%%%%%%%%%%%%%%%%%%%%%%%%%%%%%%%%%%%%%
\begin{prop}\label{prop initialcov}
%%%%%%%%%%%%%%%%%%%%%%%%%%%%%%%%%%%%%%%%%%%%%%%%%%%%%%%%%%%%%%%%
Let $G$ be a finite group and $f:X\onto G$ be a central extension of $G$ by $H_{2\setminus 1}(G)$ whose
differential $\mu(f):H_2(G)\ra H_{2\setminus 1}(G)$ is a surjection. Then:
\begin{enumerate}
\item $H_1(f):H_1(X)\ra H_1(G)$ is an isomorphism.
\item The following is an exact sequence:
\[0\ra H_2(X)\xrightarrow{H_2(f)} H_2(G)\xrightarrow{\mu(f)}H_{2\setminus 1}(G)\ra 0\]%%%%%%%.
\item $H^2(X,H_{2\setminus 1}(G))=0$.
\item The homomorphism  $f:X\onto G$   is a cellular cover (i.e.~$\text{\rm Hom}(X,f)$ is a bijection).
\item The cellular covers $f:X\onto G$  and  $e_G:E\onto G$ are equivalent.
\end{enumerate}
\end{prop}
\begin{proof}
Since $G$ is finite, $H_2(G)$ is finite and so is its quotient
$ H_{2\setminus 1}(G)$. The group $X$ is then also finite and we have the following  exact sequence:
  \[H_1(X)\otimes H_{2\setminus 1}(G)\xrightarrow{} H_2(X)\xrightarrow{H_2(f)} H_2(G)\xrightarrow{\mu(f)} H_{2\setminus 1}(G)\xrightarrow{\alpha} H_1(X)\xrightarrow{H_1(f)} H_1(G)\rightarrow 0
\]

 \noindent (1):\quad  As $\mu(f)$ is a surjection, the homomorphism $\alpha$, in the above sequence,   is trivial, and hence  $H_1(f):H_1(X)\ra H_1(G)$ is an isomorphism. This is statement (1).
\smallskip

 \noindent (2):\quad The orders of
$H_1(G)$ and  $H_{2\setminus 1}(G)$ are coprime and thus  $H_1(G)\otimes H_{2\setminus 1}(G)=0$.
Using  statement (1), we  then get  $H_1(X)\otimes H_{2\setminus 1}(G)=0$. The homomorphism
$H_2(f)$ is therefore an injection which  proves statement (2).
\smallskip

 \noindent (3):\quad By the universal
coefficient exact sequence, to show the statement, we need to prove that
$\text{Ext}^1(H_1(X),H_{2\setminus 1}(G))=0$ and $\text{Hom}(H_2(X), H_{2\setminus 1}(G))=0$.
The triviality of $\text{Ext}^1(H_1(X),H_{2\setminus 1}(G))$ follows from the fact that the orders of $H_1(X)=H_1(G)$ and  $H_{2\setminus 1}(G)$ are coprime.

Since   $H_{2\setminus 1}(G)$ is the localization $H_2(G)[S^{-1}]$, where $S$ is the set of primes that divide the
order of $H_1(G)$, the homomorphism $\mu(f)$ factors uniquely as:
\[\xymatrix{
H_2(G)\rrto^(.45){\text{localization}}\ar@/ _ 15pt/[rrr]|{\mu(f)} & & H_{2\setminus 1}(G)\rto^{h} & H_{2\setminus 1}(G).
}\]
The surjectivity of     $\mu(f)$ implies the surjectivity of   $h$. As a surjective homomorphism between finite groups, $h$ is an isomorphism.
The kernel of $\mu(f)$, which by (2) is given by $H_2(X)$, is therefore  isomorphic to the kernel of the
localization homomorphism $H_2(G)\ra H_2(G)[S^{-1}]$.
The  primes dividing the order of $H_2(X)$ are thus  among the primes dividing the order of $H_1(G)$. This means that the orders of $H_2(X)$  and $H_{2\setminus 1}(G)$ are  coprime and hence the group $\text{Hom}(H_2(X),H_{2\setminus 1}(G))$ is also trivial.
\smallskip

 \noindent (4):\quad  We need to  show $\text{Hom}(X,f):\text{Hom}(X,X)\ra \text{Hom}(X,G)$  is a bijection.
 The kernel $H_{2\setminus 1}(G)$ of $f:X\onto G$ is central in $X$. Moreover, as
 the orders of $H_{2\setminus 1}(G)$ and $H_1(X)$ are relatively prime,
 $\text{Hom}(X,H_{2\setminus 1}(G))=\text{Hom}(H_1(X),H_{2\setminus 1}(G))=0$.
 The injectivity of  $\text{Hom}(X,f)$ follows then from~\ref{prop charcCmono}.

 It remains to prove that $\text{Hom}(X,f):\text{Hom}(X,X)\ra \text{Hom}(X,G)$ is also surjective.  Let $g:X\ra G$ be an arbitrary homomorphism. Consider the following commutative diagram, where the right hand square is a pull-back square:
 \[\xymatrix{
H_{2\setminus 1}(G) \ar@{=}[d] \ar@{^(->}[r] & P\repi^{f'}\dto^{g'}  & X\dto^g\\
  H_{2\setminus 1}(G)\rto \ar@{^(->}[r]& X\repi^{f} & G
 }\]
 Note that $f':P\onto X$ represents a central extension of $X$ by $H_{2\setminus 1}(G)$.  According to statement (3) any such central extension is split.  Let $s:X\ra P$ be its section. The composition $g's:X\ra X$ is then a homomorphism for which
 $fg's=g$. This shows surjectivity of  $\text{Hom}(X,f)$.
\smallskip

 \noindent (5):\quad  The argument to show that $f:X\ra G$ and $e_G:E\ra G$ are equivalent cellular covers  is the same as in the proof of the surjectivity
 in the previous statement.  Consider the following commutative diagram, where the bottom right square is a pull-back square:
 \[\xymatrix{
 &  H_{2\setminus 1}(G) \ar @{=}[r]\dmono &  H_{2\setminus 1}(G) \dmono\\
 H_{2\setminus 1}(G) \ar @{=}[d]\rmono  & P\repi^{f'}\depi^{e'}  & E\depi^{e_G}\\
  H_{2\setminus 1}(G)\rmono & X\repi^{f} & G
 }\]
Both $e':P\onto X$ and $f':P\onto E$ represent central extensions. Statement (3)  implies that these extensions are split.
 Using their sections we can construct homomorphisms $h:X\ra E$ and $g:E\ra X$ for which $e_Gh=f$ and $fg=e_G$.
 It follows that $e_Ghg=e_G$ and $fgh=f$. As $e_G$ and $f$ are cellular  covers, we can conclude $hg=\text{id}_E$ and
 $gh=\text{id}_X$.  This proves  (5).
\end{proof}

%%%%%%%%%%%%%%%%%%%%%%%%%%%%%%%%%%%%%%%%%%%%%%%%%%%%%%%%%%
\begin{Def}
%%%%%%%%%%%%%%%%%%%%%%%%%%%%%%%%%%%%%%%%%%%%%%%%%%%%%%%%
Let $G$ be a finite group. We call the homomorphism $e_G:E\ra G$  the {\bf initial cellular cover} of $G$.
We will use the same name for the equivalence class in $\cov{G}$ represented by
$e_G:E\ra G$.
\end{Def}

%%%%%%%%%%%%%%%%%%%%%%%%%%
%%%%%%%%%%%%%%%%%%%%%%%%%%
%%%%%%%%%%%%%%%%%%%%%%%%%%
%section6
\section{Generalized subgroups of a finite group}\label{sec gensuboffinite}
%%%%%%%%%%%%%%%%%%%%%%%%%%
%%%%%%%%%%%%%%%%%%%%%%%%%%
%%%%%%%%%%%%%%%%%%%%%%%%%%%

The collection $\cov{G}$ is a subcollection of $\gensub{G}$.
Thus to show for example that $G$ has finitely many cellular covers it is  enough to show that $\gensub{G}$ is a finite set.
The aim of this section is to do that under the assumption that $G$ is a finite group.

For a homomorphism $a:X\ra G$, we use the symbol $I_a$ to denote its  image $\text{im}(a)$.
This is one of the two invariants we  use  to enumerate generalized subgroups of $G$.
Note that if generalized subgroups $a:X\ra G$ and $b:Y\ra G$ are equivalent, then  they have the same images. Thus the  function $a\mapsto I_a$ is well define on the collection  $\gensub{G}$ of equivalence classes of generalized subgroups. Furthermore it is immediate from the definition that
a homomorphism $a:X\ra G$  is a generalized subgroup of $G$ if and only
if $a:X\onto I_a$ is a generalized subgroup of $I_a$. Thus any generalized subgroup is a composition of a surjective generalized subgroup and an inclusion.
This is the  reasons why surjective generalized subgroups are important for us.

%%%%%%%%%%%%%%%%%%%%%%%%%%%%%%%%%%%%%%%%%%%%%%%%%
%6.2
\begin{Def}
%%%%%%%%%%%%%%%%%%%%%%%%%%%%%%%%%%%%%%%%%%%%%%%%%%%%%%%%%%%%%
 $\surgensub{G}$ denotes the collection of equivalence classes of generalized subgroups of $G$
 which are represented by surjective homomorphisms.
 \end{Def}
For any subgroup $I$ of $G$, let $\text{in}_I:\surgensub{I}\subset \gensub{G}$ be the function that assigns to
an equivalence class of a surjective generalized subgroup $a:X\onto I$ of $I$ the equivalence class of the composition $a:X\onto I\subset G$.
By summing up these inclusions over all
the subgroups of $G$, it is then clear that we get a bijection:
\begin{prop} The following function is a bijection:
\[\coprod_{I\subset G}\text{\rm in}_I:\coprod_{I\subset G} \surgensub{I}\ra  \gensub{G}\]
\end{prop}
To
enumerate $\gensub{G}$ it thus suffices to enumerate  $\surgensub{I}$ for all  subgroups $I$ of $G$.
We will do that under the assumption that $G$ is finite.
Let us then  assume that $G$ is {\bf finite}.

We start with defining a set
 used to enumerate the collections $\surgensub{I}$.

%%%%%%%%%%%%%%%%%%%%%%%%%%%%%%%%%%%%%%%%%%%%%%%%%%%%%%%%%%%%%%%%%%
%6.1
\begin{Def}
%%%%%%%%%%%%%%%%%%%%%%%%%%%%%%%%%%%%%%%%%%%%%%%%%%%%%%%%%%%
Let $A$ be an abelian group.
\begin{enumerate}
\item Two surjections
$\sigma:A\onto K$ and $\tau:A\onto L$ are defined to be equivalent, if there is an isomorphism $h:K\ra L$ such that $h\sigma=\tau$ (such an isomorphism, if it exists, is necessary unique).
\item The symbol $\text{Quot}(A)$ denotes the set of
equivalence classes of surjections  out of $A$.
\end{enumerate}
\end{Def}

Note that the subgroup  of $A$ given by the kernel of a surjection
$\sigma:A\onto K$ depends only on the equivalence class of $\sigma$ in $\text{Quot}(A)$. It is then clear that the function
that assigns to an element $[\sigma]$ in $ \text{Quot}(A)$ the subgroup $\text{Ker}(\sigma)$ of $A$ is a bijection
between $ \text{Quot}(A)$ and  the set of all the subgroups of $A$ which, in the case $A$ is finitely generated,  coincides with the set  $\gensub{A}$. Thus for a finitely generated abelian group $A$, we
shall identify
$ \text{Quot}(A)$ with   $\gensub{A}$.
 For example let $k$ be a positive integer.  The element of  $ \text{Quot}(A)$ that corresponds to the $k$-torsion subgroup
of $A$ is denoted by $q_k$. It is represented by the surjection, denoted by the same symbol
\[
q_k:A\onto  A/(k\text{-torsion}),
\]
 that maps an element $a\in A$ to its coset.
In the case of    the cyclic group ${\mathbf Z}/n$ ($n>0$), these are all the elements of
$ \text{Quot}({\mathbf Z}/n)$.
For any $k> 0$ dividing $n$, the  $k$-torsion subgroup of ${\mathbf Z}/n$ is the subgroup generated by $n/k$.
It is the unique subgroup   isomorphic to ${\mathbf Z}/k$.    In this way $ \text{Quot}({\mathbf Z}/n)$ is in bijection with the set of all positive divisors of $n$.

From now until Definition \ref{def In(G)} we
{\bf fix a subgroup $I$ of $G$}.
We enumerate $\surgensub{I}$ using the set $\text{\rm Quot}(H_{2\setminus 1}(I))$
(recall that $H_{2\setminus 1}(I)$ denotes the localization $H_2(I)[S^{-1}]$,
where $S$ is the set of primes dividing the order of $H_1(I)$, see~\ref{def 2minus1}).
To do that we  define two functions:
\[\mu:\surgensub{I}\ra\text{\rm Quot}(H_{2\setminus 1}(I)),\ \ \ \ \ \  \ \ \ \  \Psi:\text{\rm Quot}(H_{2\setminus 1}(I))\ra \surgensub{I},\]
and show that their    compositions $\mu\Psi$ and $\Psi\mu$ are the identities. For a surjective
generalized subgroup $a:X\onto I$, the value $\mu(a)\in \text{\rm Quot}(H_{2\setminus 1}(I))$ is called the differential of $a$.
Recall that according to~\ref{prop charcCmono}, the kernel $K_a:=\text{Ker}(a)$ of $a$ is central in $X$.
Thus the homomorphism $a:X\onto I$
represents a central  extension of $I$ by $K_a$. We  use the corresponding element in
$H^2(I, K_a)$ to define the differential. First we  need:

%%%%%%%%%%%%%%%%%%%%%%%%%%%%%%%%%%%%%%%%%%%%%%%%%%%%%%%%%%%%%%%%%%%%
%6.3
\begin{prop}\label{prop propgsub}
%%%%%%%%%%%%%%%%%%%%%%%%%%%%%%%%%%%%%%%%%%%%%%%%%%%%%%%%%%%%%
Let  $a:X\onto I$ be a surjective generalized subgroup of a finite group $I$. Then:
\begin{enumerate}
\item $H_1(a):H_1(X)\ra H_1(I)$ is an isomorphism.
\item $\text{\rm Ext}^1\left(H_1(I),K_a\right)=H_1(X)\otimes K_a=0$.
\item $\mu:H^2(I,K_a)\ra \text{\rm Hom}\left(H_2(I),K_a\right)$ is an isomorphism.
\item $0\ra H_2(X)\xrightarrow{H_2(a)} H_2(I)\xrightarrow{\mu(a)} K_a\ra 0$ is an exact sequence.
\item If $Y\subset X$ is a subgroup such that $a(Y)=I$, then $Y=X$.
\end{enumerate}
\end{prop}
\begin{proof}
Since $G$ is finite, by~\ref{prop presofsolnilfin}(2),  $X$  is also finite and we have en exact sequence
of the central extension $a:X\onto G$:
\[H_1(X)\otimes K_a\xrightarrow{} H_2(X)\xrightarrow{H_2(a)} H_2(I)\xrightarrow{\mu(a)} K_a\xrightarrow{\alpha} H_1(X)\xrightarrow{H_1(a)} H_1(I)\rightarrow 0
\]

 \noindent (1):\quad
The finiteness of  $G$ implies also that  $K_a\subset [X,X]$ (see~\ref{prop presofsolnilfin}(2)). The homomorphism
$\alpha: K_a\ra H_1(X)$, in the above  exact sequence, is then  trivial
and  $H_1(a)$ is an isomorphism. This proves (1).
\smallskip

 \noindent (2):\quad
 According to~\ref{prop charcCmono}(2), $\text{Hom}\left(H_1(X),K_a\right)=\text{Hom}(X,K_a)=0$.
The orders of $H_1(X)$ and $K_a$  are therefore relatively prime numbers.
As $H_1(X)$ and $H_1(I)$ are isomorphic (statement (1)), we get
$\text{\rm Ext}^1\left(H_1(I),K_a\right)=H_1(X)\otimes K_a=0$ which  is statement (2).
\smallskip

 \noindent (3):\quad
This is a consequence of   the universal coefficient exact sequence and the triviality of
$\text{\rm Ext}^1\left(H_1(I),K_a\right)$
(statement (2)).
\smallskip

 \noindent (4):\quad
This follows from  the exact sequence of the central extension $a:X\onto G$ above and  the triviality of
$H_1(X)\otimes K_a$ (statement (2)).
\smallskip

 \noindent (5):\quad
We have $X=YK_a$ and since $K_a$ is central in $X$ we get that
$[X,X]=[Y,Y]$.  However, as we observed earlier in the proof,
$K_a\subset [X,X]$ and it follows that $K_a\subset Y$ and so $Y=X$.
\end{proof}
\smallskip

\noindent
{\bf The differential of a surjective generalized subgroup.}
%notechange
If $a:X\onto  I$ is a surjective generalized subgroup,
then according to~\ref{prop propgsub}(3),  the homomorphism $\mu:H^2(I,K_a)\ra \text{\rm Hom}(H_2(I),K_a)$ is an isomorphism. The extension  $a:X\onto I$,
which is an element of $H^2(I, K_a)$, can be then identified with the homomorphism $\mu(a):H_2(I)\ra K_a$.
According to~\ref{prop propgsub}(4) such  homomorphisms associated with generalized subgroups  are surjections.
Furthermore, as $H_1(I)\otimes K_a=0$ (see~\ref{prop propgsub}(2)), the primes that divide the order of
$H_1(I)$ do not divide the order of $K_a$. This means that the localization
$K_a\ra K_a[S^{-1}]$ is an isomorphism, where $S$ is the set of primes that divide the order of $H_1(I)$. Consequently
$\mu(a):H_2(I)\onto K_a$ factors uniquely as:
\[\xymatrix{
H_2(I)\rrto^(.33){\text{localization}}\ar@{->>}@/ _ 15pt/[rrr]|{\mu(a)} && H_2(I)[S^{-1}]=H_{2\setminus 1}(I_a)\rto & K_a.
}\]
We will use  the same symbol $\mu(a): H_{2\setminus 1}(I)\onto K_a$ to denote the surjection in the
above factorization.
We can now define:

%%%%%%%%%%%%%%%%%%%%%%%%%%%%%%%%%%%%%%%%%%%%%%%%%%%%
%6.5
\begin{Def}
%%%%%%%%%%%%%%%%%%%%%%%%%%%%%%%%%%%%%%%%%%%%%%%%%%%%%%%
Let  $a:X\onto I$ be a surjective generalized subgroup of $I$. The
element in
$\text{Quot}(H_{2\setminus 1}(I))$ represented by  the surjection
$\mu(a): H_{2\setminus 1}(I)\onto K_a$ is called the deferential of $a$ and is denoted also
by the same symbol  $\mu(a)$.
\end{Def}

Assume now that  $a:X\onto I$ and $b:Y\onto I$ are equivalent surjective generalized subgroups of $I$ and $h:X\ra Y$ is  an isomorphism for which $bh=a$.   By the naturality of the exact sequence of  a central extension, we get a commutative diagram with exact rows:
\[\xymatrix{
0\rto & H_2(X)\rto^{H_2(a)}\dto_{H_2(h)} & H_2(I)\rto^{\mu(a)}\dto^{\text{id}} & K_a\rto \dto^{h} &0\\
0\rto & H_2(Y)\rto^{H_2(b)} & H_2(I)\rto^{\mu(b)} & K_b\rto &0
}\]
After localizing with respect to the set $S$ of primes that divide the order of $H_1(I)$, we get then
that $\mu(b):H_{2\setminus 1}(I)\onto K_b$ is the composition of $\mu(a):H_{2\setminus 1}(I)\onto K_a$ and
the isomorphism $h:K_a\ra K_b$. The surjections  $\mu(a)$ and $\mu(b)$ are thus equivalent and define the same element in $\text{Quot}(H_{2\setminus 1}(I))$.
It follows that the differential  is well defined on the collection  $\surgensub{I}$ of equivalence classes of generalized
subgroups.   In this way we get a well-defined function:
\[\xymatrix{
\surgensub{I}\ni [a:X\ra I]\ar@{|->}[rr]^{\mu} & &  \mu(a)\in \text{Quot}(H_{2\setminus 1}(I)).
}\]
which we also  denote by $\mu$.

Next we define   an inverse to $\mu$ (see \ref{thm classificationsurgsub}), which we denote by  $\Psi:\text{Quot}(H_{2\setminus 1}(I)) \ra\surgensub{I}$.
 Let us choose a surjection  $\sigma:H_{2\setminus 1}(I)\onto  K$ that represents  a given element in
$\text{\rm Quot}(H_{2\setminus 1}(I))$.
Recall that $e_I:E\onto I$ denotes the initial central extension of $I$ by $H_{2\setminus 1}(I)$ (see Section~\ref{sec initialcover}).
Define:
\[X:=\text{colim}(\xymatrix{K&  H_{2\setminus 1}(I)\ar@{->>}[l]_(.55){\sigma}\ar@{^(->}[r]& E}).\]
% \[X:=\text{colim}\left(\xymatrix{K&\xleftarrow{\sigma} H_{2\setminus 1}(I)\subset E}\right)\]
This is just $E$ divided by the kernel of the map
$H_{2\setminus 1}(I)\twoheadrightarrow K$.
Let $a:X\onto  I$ be  the homomorphism that fits into the following commutative diagram where
 $\pi:E\ra X$ is the structure map of the colimit:
 \[\xymatrix{
  H_{2\setminus 1}(I)\ar@{^(->}[r]\depi_{\sigma} & E\repi^{e_I}\dto^{\pi} & I\ar@{=}[d]^{\text{id}}\\
 K \ar@{^(->}[r] & X\repi^{a} &I
 }\]
 Note that $a:X\onto I$ is a central extension of $I$ by $K$.  By the naturality of the exact sequence of
 a central extension we get a commutative diagram of homology groups:
 \[\xymatrix{
 H_2(I)\rto^(.45){\mu(e_I)}\dto_{\text{id}} &  H_{2\setminus 1}(I)\rto\depi^{\sigma} & H_1(E)\rto^{H_1(e_I)}\dto & H_1(I)\rto\dto^{\text{id}}  & 0\\
 H_2(I)\rto^{\mu(a)} &  K\rto^(.42){\alpha} & H_1(X)\rto^{H_1(a)} & H_1(I)\rto & 0
} \]
As $\mu(e_I)$ and $\sigma$ are  surjections, then so is $\mu(a)$. The homomorphism  $\alpha$ is therefore trivial and
consequently $H_1(a):H_1(X)\ra H_1(I)$ is an isomorphism. Since $K$ is a quotient of  $H_{2\setminus 1}(I)$,
the primes that divide the order of $H_1(X)$ do not divide the order of $K$.
It follows that $\text{Hom}(X,K)=\text{Hom}(H_1(X),K)=0$.    We define  $\Psi([\sigma])$ to be the element of $\surgensub{I}$ given by the equivalence class represented by this surjective
 generalized subgroup  $a:X\onto I$. It is straight forward to check that $\Psi([\sigma])$ does not depend on the choice of a surjection  $\sigma:H_{2\setminus 1}(I)\onto K$ representing the given element in $  \text{\rm Quot}(H_{2\setminus 1}(I))$.  In this way we have a well defined function $\Psi:\text{\rm Quot}(H_{2\setminus 1}(I))\ra \surgensub{I}$.
Note further that   $\sigma =\mu(a)$. This means that $\mu\Psi=\text{id}$.

To show that $\Psi\mu$ is also the identity, let us choose a surjective  generalized subgroup $a:X\onto I$.
Let $b:Y\onto I$ be a surjective generalized subgroup representing  $\Psi\mu(a)$. We need to show that
$a$ and $b$ are equivalent.
Consider the pair $\mu(a)$.
Since
$\mu\Psi$ is the identity:
 \[\mu(b)=\mu\Psi\mu(a)=\mu(a)\]
 This means that  the differential $\mu(b)$ is equivalent
to $\mu(a)$.
As in this case the differential determines
the central extension it comes from, $a$ and $b$    are indeed
equivalent generalized subgroups.
We just have shown:

%%%%%%%%%%%%%%%%%%%%%%%%%%%%%%%%%%%%%%%%%%%%%%%%%%%%%%%%%%%%%%%%%%%%%%%%
%6.6
 \begin{thm}\label{thm classificationsurgsub}
 %%%%%%%%%%%%%%%%%%%%%%%%%%%%%%%%%%%%%%%%%%%%%%%%%%%%%%%%%%%%%%%%%
Let $I$ be a finite group. The function $\mu:\surgensub{I}\ra\text{\rm Quot}(H_{2\setminus 1}(I))$ is a bijection.
 \end{thm}

For reference we record the equalities  $\mu\Psi=\text{id}$ and $\Psi\mu=\text{id}$  in the form of:

\begin{prop}\label{prop mupsiid}
Any  surjective generalized subgroup $a:X\onto I$  fits into  the following commutative ladder of short exact sequences with the
left square being a push-out:
 \[\xymatrix{
  H_{2\setminus 1}(I)\ar@{^(->}[r]\depi_{\mu(a)} & E\repi^{e_I}\dto^{\pi} & I\ar@{=}[d]^{\text{\rm id}}\\
K_a \ar@{^(->}[r] & X\repi^{a} &I
 }\]
\end{prop}
%notechange

Theorem~\ref{thm classificationsurgsub}  can be used  to enumerate all the generalized subgroups of $G$.

%%%%%%%%%%%%%%%%%%%%%%%%%%%%%%%%%%%%%%%%%%%%%%%%%%%%%%%%
\begin{Def}\label{def In(G)}
%%%%%%%%%%%%%%%%%%%%%%%%%%%%%%%%%%%%%%%%%%%%%%%%%%%%%%%%
$\text{\rm In}(G)$ is defined to be  the set of pairs $(I,\sigma)$ where $I$ is a subgroup of $G$ and
$\sigma \in\text{\rm Quot}(H_{2\setminus 1}(I))$.
\end{Def}
As a corollary of~\ref{thm classificationsurgsub} we get:

\begin{cor}\label{cor classificationgsub}
For a finite group $G$, the following function is a bijection between $\gensub{G}$ and $\text{\rm In}(G)$:
\[\xymatrix{
\gensub{G}\ni [a:X\ra G]\ar@{|->}[rr] & &  (I_a, \mu(a:X\onto I_a))\in \text{\rm In}(G)
}
\]
%The function that assigns to an element in  $\gensub{G}$ represented by $a:X\ra G$ a pair $(I_a, \mu(a:X\onto I_a))\in \text{\rm In}(G)$
%is a bijection between
%$\gensub{G}$ and $\text{\rm In}(G)$.
\end{cor}

Since  $\cov{G}$ is a subcollection of $\gensub{G}$ and $\text{\rm In}(G)$ is finite,~\ref{cor classificationgsub}
implies:
\begin{cor}\label{cor covidemfinite}
For a finite group $G$, the collections $\gensub{G}$,  $\cov{G}$, and $\text{\rm Idem}(G)$  are finite sets.
\end{cor}

%%%%%%%%%%%%%%%%%%%%%%
%%%%%%%%%%%%%%%%%%%%%%
%%%%%%%%%%%%%%%%%%%%%%
%section7
\section{Surjective cellular covers of finite groups}
%%%%%%%%%%%%%%%%%%%%%%
%%%%%%%%%%%%%%%%%%%%%%
%%%%%%%%%%%%%%%%%%%%%%
Recall that $G$ is assumed to be {\bf  finite}.
According to~\ref{cor classificationgsub}, a generalized subgroup  of  such a group  $G$ is determined  by two invariants: its image  and its
 differential.   For a given subgroup $I\subset G$ the differential classify all possible generalized subgroups of $G$ whose image is $I$, or equivalently generalized subgroup of $I$ which are represented by surjective homomorphisms.  Thus to classify cellular covers  we need to determined first the subgroups of $G$ which are images of cellular covers and then, for any such subgroup $I$, identify these surjective generalized subgroups of $I$ which
 are cellular covers of $G$.  Unfortunately we can not say much about the first step in this process. We do not know how to identify subgroups of $G$ which are images of
 cellular covers. However we can deal with the second step: the cellular covers of finite
  groups which are represented by surjective homomorphisms  in several important cases. This is the aim of this section.

%%%%%%%%%%%%%%%%%%%%%%%%%%%%%%%%%%%%%%%%%%%%%%%%%%%%%%%%%
%7.1
\begin{Def}
%%%%%%%%%%%%%%%%%%%%%%%%%%%%%%%%%%%%%%%%%%%%%%%%%%%%%%%
%\begin{enumerate}
%\item $\surgensub{G}$ denotes the collection of generalized subgroups of $G$ that are represented by  surjective homomorphisms.
%\item
$\surcov{G}$ denotes the collection of cellular covers of $G$ that are represented by
 surjective homomorphisms.
%\end{enumerate}
\end{Def}

\noindent
%%%%%%%%%%%%%%%%%%%%%%%%%%%%%%%%%%%%%%%%%
%
{\bf $H_{2\setminus 1}$ as a functor.}
%%%%%%%%%%%%%%%%%%%%%%%%%%%%%%%%%%%%%%%
To describe surjective covers we  use
 functorial properties of $H_{2\setminus 1}$ (see~\ref{def 2minus1}).
Any homomorphism $f:X\ra G$ induces a homomorphism
of the Schur multipliers $H_2(f):H_2(X)\ra H_2(G)$ (the Schur multiplier is a functor). In general this homomorphism $H_2(f)$
however does not induce a homomorphism between   $H_{2\setminus 1} (X)$ and  $H_{2\setminus 1}(G).$
For that we need an  additional  assumption on  $X$ and $ G$.
We need to  assume that both $X$ and $G$ are finite and
 that the set $S_{X}$ of primes that divide the order of $H_1(X)$ is a {\bf subset}  of the set $S_{G}$ of primes that
divide the order of  $H_1(G)$.  In this case there is a unique homomorphism:
\[H_{2\setminus 1}(f):H_{2\setminus 1}(X)=H_2(X)[S_X^{-1}]\ra H_2(G)[S_G^{-1}]=H_{2\setminus 1}(G)\]
 for which the following diagram commutes:
\[\xymatrix{
H_2(X)\rrto^(.4){\text{localization}} \dto_{H_2(f)}& & H_2(X)[S_X^{-1}]\ar@{=}[r] & H_{2\setminus 1}(X)\dto^{H_{2\setminus 1}(f)}\\
H_2(G)\rrto^(.4){\text{localization}} & & H_2(G)[S_G^{-1}]\ar@{=}[r] & H_{2\setminus 1}(G)
}\]
This is because $H_{2\setminus 1}(G)$ is uniquely divisible by the primes in $S_{X}$.
Observe further the uniqueness implies  $H_{2\setminus 1}(fg)=H_{2\setminus 1}(f)H_{2\setminus 1}(g)$, for any
two homomorphisms  $g:Y\ra X$ and $f:X\ra G$ for which the inclusions $S_Y\subset S_X\subset S_G$
of  sets of prime numbers that divide the corresponding orders of the abelianizations hold.

We conclude that $H_{2\setminus 1}$  is a functor on the full subcategory of finite groups
with a fixed isomorphism type of $H_1$.  For example
let  $X$ and $G$ be finite groups for which there is a surjective homomorphism $c:X\onto G$
which is a generalized subgroup. By~\ref{prop propgsub}(1),
$H_1(X)$ and $H_1(G)$ are isomorphic. Thus, for such groups $X$ and $G$, any homomorphism $f:X\ra G$ induces  $H_{2\setminus 1}(f):H_{2\setminus 1}(X) \ra H_{2\setminus 1}(G)$.

Since $G$ is finite, according to~\ref{prop propgsub}(4),
 a surjective generalized subgroup   $c:X\onto G$ induces an exact sequence:
 \[0\ra H_2(X)\xrightarrow{H_2(c)} H_2(G)\xrightarrow{\mu(c)} K_c\ra 0.\]
After localization with respect to the set $S$ of  primes that divide the order of $H_1(X)=H_1(G)$
(see~\ref{prop propgsub}(1)),
we get again an exact sequence:
 \[0\ra H_{2\setminus 1}(X)\xrightarrow{H_{2\setminus 1}(c)} H_{2\setminus 1}(G)\xrightarrow{\mu(c)} K_c\ra 0.\]
 Thus the kernel of the differential $\mu(c):H_{2\setminus 1}(G)\onto K_c$ is given by the
 image of $H_{2\setminus 1}(c):H_{2\setminus 1}(X)\subset H_{2\setminus 1}(G)$ which we simply denote by
 $H_{2\setminus 1}(X)$.

To enumerate the set   $\surcov{G}$  of  surjective  cellular covers of $G$ we need to understand for which  surjective generalized subgroups $c:X\onto G$
the function $\text{Hom}(X,c):\text{Hom}(X,X)\ra \text{Hom}(X,G)$ is a surjection. We start by
determining the image of $\text{Hom}(X,c)$. This image consists of   homomorphisms $f:X\ra G$ that can be lifted through $c$
and expressed as  compositions of some $s:X\ra X$ and $c:X\ra G$. The following proposition describes such homomorphisms:

\begin{prop}
Let $c:X\onto G$  be a surjective generalized subgroup.  A homomorphism $f:X\ra G$ can be expressed as a composition of $s:X\ra X$ and $c:X\onto G$ if and only if the image of $H_{2\setminus 1}(f)$
lies in the image $H_{2\setminus 1}(c)$.
\end{prop}
\begin{proof}
Clearly if $f=cs$ for some  $s:X\ra X$, then $H_{2\setminus 1}(f)=H_{2\setminus 1}(c)H_{2\setminus 1}(s)$ and hence the image of $H_{2\setminus 1}(f)$ is in the image of $H_{2\setminus 1}(c)$.

Assume that  $H_{2\setminus 1}(f)$ is in the image of $H_{2\setminus 1}(c)$.
Recall that  $c:X\onto G$ fits into the following commutative diagram where the left bottom
square is a push-out square and $e_G:E\onto G$ is   the initial extension (see~\ref{prop mupsiid}):
 \[\xymatrix{
H_{2\setminus 1}(X)\ar@{_(->}[d]_{H_{2\setminus 1}(c)}\ar@{=}[r]^{\text{id}}& H_{2\setminus 1}(X)\ar@{_(->}[d]\\
H_{2\setminus 1}(G)\ar@{^(->}[r]\depi_{\mu(c)} & E\repi^{e_G}\ar@{->>}[d]^{\pi} & G\ar@{=}[d]^{\text{id}}\\
K_c \ar@{^(->}[r] & X\repi^{c} & G
 }\]

To prove the lemma we will construct the following commutative diagram:

%%%%%%%%%%%
%%%%%%%%%%%%

%notechange
\[\xymatrix{
&{\text{Ker}(\pi)}\ar@{=}[r] \ar[d]_{f'} &H_{2\setminus 1}(X)\ar@{^(->}[r]\dto^{H_{2\setminus 1}(f)}
& E\repi^{\pi}\dto^{\bar f} & X\drto^{f}\dto^{s}\\
&H_{2\setminus 1}(X)\ar@{^(->}[r]^(0.52){H_{2\setminus 1}(c)}  &  H_{2\setminus 1}(G)\ar@{^(->}[r] & E\ar@/ _ 15pt/[rr]|{e_G}\repi^{\pi}
& X\repi^{c} & G
}\]
%%%%%%%%%%%
%%%%%%%%%%

Since  $e_G:E\onto G$ is a cellular cover (see~\ref{prop initialcov}(4)), there is a unique
$\bar f$ for which $e_G {\bar f}= f\pi$.  By the  naturality of the differentials we have a commutative square:
\[\xymatrix{
H_2(X)\ar@{->>}[r]^{\mu(\pi)}\dto_{H_2(f)} &  H_{2\setminus 1}(X)\dto^{\bar f}\\
H_2(G)\ar@{->>}[r]^{\mu(e_G)} &  H_{2\setminus 1}(G)
}
\]
This implies that  the restriction of $\bar f$ to $ H_{2\setminus 1}(X)$ is given
by $ H_{2\setminus 1}(f)$.  By the assumption the image of  $H_{2\setminus 1}(f)$ lies in the image
of  $H_{2\setminus 1}(c)$ which gives the homomorphism $f'$.
As $\bar f:E\ra E$ maps the kernel $\text{Ker}(\pi)\subset E$ to itself, we get the desired
homomorphism $s:X\ra X$.
\end{proof}

\begin{cor}\label{cor gscov}
A surjective generalized subgroup $c:X\onto G$ is a cellular cover if and only if, for any homomorphism
$f:X\ra G$, the image of $ H_{2\setminus 1}(f):H_{2\setminus 1}(X)\ra H_{2\setminus 1}(G)$ lies in
the image of $ H_{2\setminus 1}(c)$.
\end{cor}

We can use this corollary  to prove:
%%%%%%%%%%%%%%%%%%%%%%%%%%%%%%%%%%%%%%%%%%%%%%%%%%%%%%%%%%%%%%%%%%%%%%%%%%%%%%
\begin{prop}\label{prop cyclicH2}\hspace{1mm}
%%%%%%%%%%%%%%%%%%%%%%%%%%%%%%%%%%%%%%%%%%%%%%%%%%%%%%%%%%%%%%%%%%%%%%%%%%%%%%
\begin{enumerate}
\item Let $k$ be a positive divisor of the order of $H_{2\setminus 1}(G)$. If $c:X\onto G$ is a surjective generalized  subgroup whose differential is given by the quotient homomorphism $q_k:H_{2\setminus 1}(G)\ra H_{2\setminus 1}(G)/(k\text{-torsion})$, then $c$ is a cellular cover.
\item If $H_{2\setminus 1}(G)$ is cyclic, then all surjective generalized subgroups of $G$ are cellular covers.
\end{enumerate}
\end{prop}
\begin{proof}
\noindent
(1):\quad
Recall that we have an exact sequence:
\[
0\ra  H_{2\setminus 1}(X)\xrightarrow{ H_{2\setminus 1}(c)}  H_{2\setminus 1}(G)\xrightarrow{\mu(c)} H_{2\setminus 1}(G)/(k\text{-torsion}) \ra 0.
\]
Thus,  $H_{2\setminus 1}(X)$ is the $k$-torsion subgroup of $ H_{2\setminus 1}(G)$.
Let $f:X\ra G$ be a homomorphism. Since $H_{2\setminus 1}(X)$ is $k$-torsion, the image of
$H_{2\setminus 1}(f):H_{2\setminus 1}(X)\ra  H_{2\setminus 1}(G)$  lies in the $k$-torsion subgroup of
$H_{2\setminus 1}(G)$ which is the image of $H_{2\setminus 1}(c)$.
According to~\ref{cor gscov} $a$ is a cellular cover.
\smallskip

\noindent
(2):\quad If $H_{2\setminus 1}(G)$ is cyclic, then any of  its subgroups is  the $k$-torsion subgroup for some $k$.
Statement (1) follows then from statement (2).
\end{proof}
\smallskip

\noindent
{\bf The action of $\text{Out}(G)$.}
According to~\ref{thm classificationsurgsub} a surjective generalized subgroup $c:X\onto G$
is  determined  by  its  differential $\mu(c):H_{2\setminus 1}(G)\onto K_c$  which in turn is determined by its kernel
$H_{2\setminus 1}(X)$.  The  following functions  are bijections:

\[\xymatrix@R-20pt{
\surgensub{G}\rto  & \text{Quot}(H_{2\setminus 1}(G)) \rto &  \gensub{H_{2\setminus 1}(G)}; \\
{(c:X\onto G)}\ar@{|->}[r]&  {[\mu(c):H_{2\setminus 1}(G)\onto K_c]}\ar@{|->}[r] &H_{2\setminus 1}(X).
 }\]
 In this way  surjective generalized subgroups of $G$ are enumerated by  subgroups of $H_{2\setminus 1}(G)$.
To enumerate the set $\surcov{G}$,
we need to identify these elements  in $\text{Quot}(H_{2\setminus 1}(G)) $, or equivalently in $\gensub{H_{2\setminus 1}(G)} $, which are differentials of   surjective cellular covers.
For that we look at the action of   $\text{Out}(G)$ on
these sets. Let $h:G\ra G$ be an automorphism.  Consider the induced isomorphism
$H_2(h):H_2(G)\ra H_2(G)$ and its localization $H_{2\setminus 1}(h):H_{2\setminus 1}(G)\ra H_{2\setminus 1}(G)$ with respect to the set $S$ of primes that divide the order of
$H_1(G)$.  The function:
\[\xymatrix{ \gensub{H_{2\setminus 1}(G)}\times \text{Aut}(G)\ni (H, h)\ar@{|->}[rr] &&  H_{2\setminus 1}(h)^{-1}(H)\in \gensub{H_{2\setminus 1}(G)}}\]
defines a right action of  $\text{Aut}(G)$ on the set $\gensub{H_{2\setminus 1}(G)}$ of all subgroups
of $H_{2\setminus 1}(G)$.
 Since inner automorphisms induce the identity
on homology, this action induces an action of $\text{Out}(G)$ on $\gensub{H_{2\setminus 1}(G)}$.

The corresponding right action of $\text{Out}(G)$ on $\text{Quot}(H_{2\setminus 1}(G))$
can be described as follows.  Let $h: G\ra G$ be an automorphism. For a surjective homomorphism $\sigma:H_{2\setminus 1}(G)\onto K$, the composition
$\sigma H_{2\setminus 1}(h):H_{2\setminus 1}(G)\ra K$ is also surjective. Note further if  $\sigma:H_{2\setminus 1}(G)\onto K$
and $\tau:H_{2\setminus 1}(G)\onto K$ define the same element in  $\text{Quot}(H_{2\setminus 1}(G))$,
then so do their compositions $\sigma H_{2\setminus 1}(h)$ and $\tau H_{2\setminus 1}(h)$.
The following induced function defines a right action of $\text{Out}(G)$ on
$\text{Quot}(H_{2\setminus 1}(G))$:
\[ \xymatrix{
\text{Quot}(H_{2\setminus 1}(G))\times \text{Out}(G)\ni ([\sigma], [h])\ar@{|->}[rr] && \left[\sigma H_{2\setminus 1}(h)\right]\in \text{Quot}(H_{2\setminus 1}(G))}\]
Moreover the bijection that assigns to an element $\sigma$ in $\text{Quot}(H_{2\setminus 1}(G))$ its kernel $\text{Ker}(\sigma)$, which is an element in $\gensub{H_{2\setminus 1}(G)}$, is an equivariant isomorphism.

We will be interested in the fixed points of these actions.
The reason for this is:
\begin{prop}\label{prop diffcovfix}
If $c:X\onto G$ is a  surjective cellular cover, then its
differential $[\mu(c):H_{2\setminus 1}(G)\onto K_c]\in
\text{\rm Quot}(H_{2\setminus 1}(G))$  and $H_{2\setminus 1}(X)\in \gensub{H_{2\setminus 1}(G)}$
are fixed by the action of $\text{\rm Out}(G)$.
\end{prop}

\begin{proof}
Let $h:G\ra G$ be an automorphism. Since $c:X\onto G$ is a cellular cover, there is a unique
homomorphism $h':X\ra X$ that fits into the following commutative square:
\[\xymatrix{
X\rto^{h'}\depi_{c} & X\depi^{c}\\
G\rto^{h} & G
}\]
By the naturality of the differential we then get an induced commutative square:
\[\xymatrix{
H_{2\setminus 1}(G)\depi_{\mu(c)}\rto^{H_{2\setminus 1}(h)} & H_{2\setminus 1}(G)\depi^{\mu(c)}\\
K_c\rto^{h'} & K_c
}\]
It follows that $H_{2\setminus 1}(h)$ maps the kernel  of $\mu(c)$ to
itself.
This means that, as an element  of $\gensub{H_{2\setminus 1}(G)}$, this kernel
is invariant under  the action of $\text{\rm Out}(G)$.
\end{proof}

\begin{cor}\label{cor cyclictrivialaction}\hspace{2mm}
\begin{enumerate}
\item Let $k>0$ be a divisor of the exponent of $H_{2\setminus 1}(G)$. Then
the $k$-torsion subgroup of $H_{2\setminus 1}(G)$ is fixed by $\text{\rm Out}(G)$.
\item If $H_{2\setminus 1}(G)$ is cyclic, then  the action of  $\text{\rm Out}(G)$ on  the sets $\text{\rm Sub}(H_{2\setminus 1}(G))$ and $\text{\rm Quot}(H_{2\setminus 1}(G))$ is trivial.
\end{enumerate}
\end{cor}

Let $\invquot{H_{2\setminus 1}(G)}\subset \text{\rm Quot}(H_{2\setminus 1}(G))$ and
$\invsub{H_{2\setminus 1}(G)}\subset \gensub{H_{2\setminus 1}(G)} $  be the fixed points of the
action of $\text{Out}(G)$.
According to what has been  proven,  we have the following sequence of inclusions:

\[\xymatrix@R-10pt{
k \ar@{|->}[dd] & &*=<0pt, 20pt>{\{k\in{\mathbf Z}\ |\ k>0\text{ and } k\text{ divides the order of  }H_{2\setminus 1}(G)\}}
\ar@{_(->}[d]
\\
&& \surcov{G} \ar@{_(->}@/_10pt/[dl]_{\mu}\ar@{^(->}@/^10pt/[dr] \\
q_k &\invquot{H_{2\setminus 1}(G)}\ar@{^(->}@/_10pt/[dr] & & \surgensub{G}
\ar@{=}@/^10pt/[dl]_(.4){\mu}\\
& &  \text{\rm Quot}(H_{2\setminus 1}(G))
}\]
%%%%%%%%%%%%%%%%%%%%%%%%%%%%%%%%
%%%%%%%%%%%%

Proposition~\ref{prop cyclicH2} can be rephrased as:
\begin{cor}\label{cor sumcyclicH2min1}
Let $G$ be a group for which $H_{2\setminus 1}(G)$ is cyclic.
Then all the inclusions in the above diagram are bijections. In particular:
\begin{enumerate}
\item The differential $\mu:\surcov{G}\ra \text{\rm Quot}(H_{2\setminus 1}(G))$ is a bijection.
\item Any surjective generalized subgroup of $G$ is a cellular cover.
\item Let $c:X\onto G$ be a surjection. Then $\text{\rm Hom}(X,c):\text{\rm Hom}(X,X)\ra \text{\rm Hom}(X,G)$
is a bijection if and only if it is an injection.
\end{enumerate}
\end{cor}

%%%%%%%%%%%%%%%%%%%%%%%%%%%%%%%%%%%%%%%%%%%%%%%%%%%%%
%%%%%%%%%%%%%%%%%%%%%%%%%%%%%%%%%%%%%%%%%%%%%%%%%%%%%
%%%%%%%%%%%%%%%%%%%%%%%%%%%%%%%%%%%%%%%%%%%%%%%%%%%%%
%section8
\section{Cellular covers of finite simple groups}
%%%%%%%%%%%%%%%%%%%%%%%%%%%%%%%%%%%%%%%%%%%%%%%%%%%%%%
%%%%%%%%%%%%%%%%%%%%%%%%%%%%%%%%%%%%%%%%%%%%%%%%%%%%%%
%%%%%%%%%%%%%%%%%%%%%%%%%%%%%%%%%%%%%%%%%%%%%%%%%%%%%%
The aim of this section is to classify cellular covers of finite simple groups.
A simple group $G$ has a trivial abelianization and thus  $H_{2\setminus 1}(G)=H_{2}(G)$.
According  to~\ref{prop diffcovfix}  the differential  induces an inclusion
$\mu:\surcov{G}\subset \invquot{H_{2}(G)}$.
Our key result is:

%%%%%%%%%%%%%%%%%%%%%%%%%%%%%%%%%%%%%%%%%%%%%%%%%%%%
%8.1
\begin{thm}\label{thm surcovfs}
%%%%%%%%%%%%%%%%%%%%%%%%%%%%%%%%%%%%%%%%%%%%%%%%%%%%%
If  $G$ is a  finite  simple group, then  $\mu:\surcov{G}\subset \invquot{H_{2}(G)}$ is a bijection.
\end{thm}

%%%%%%%%%%%%%%%%%%%%%%%%%%%%%%%%%%%%%%%%%%%%%%%%%%%%%%
%8.2
\begin{cor}\label{cor covsimple}
%%%%%%%%%%%%%%%%%%%%%%%%%%%%%%%%%%%%%%%%%%%%%%%%%%%%%%%%%
Let $G$ be a finite simple group. Then
the sets $\cov{G}$ and  $\text{\rm Idem}(G)$  are in bijection with
$\{0\}\coprod \text{\rm Inv}\gensub{H_{2}(G)}$.
\end{cor}
\begin{proof}
Recall that according to~\ref{prop identifyingidemp},  the sets $\cov{G}$  and
$\text{\rm Idem}(G)$ are in bijection with each other.
Let $c:X\ra G$ be a cellular cover. Since the image of $c$ is a normal subgroup of $G$,
this image is either the trivial group or the whole $G$. In the first case $X$ has to be trivial.
In the second case $c$ is a surjective cellular cover of $X$.  Thus according to~\ref{thm surcovfs}
the assignment that maps the trivial cellular cover to the element $0$ and a surjective cellular cover $c:X\onto G$
to the  image of $H_{2}(c):H_{2}(X)\subset H_{2}(G)$, is the desired bijection between
$\cov{G}$ and $\{0\}\coprod \text{\rm Inv}\gensub{H_{2}(G)}$.
\end{proof}
The key property of finite simple groups used to prove the above theorem is:

%%%%%%%%%%%%%%%%%%%%%%%%%%%%%%%%%%%%%%%%%%%%%%%%%%%%%%%%%%%%%%%%
%8.3
\begin{lemma}\label{lem basicsimplgr}
%%%%%%%%%%%%%%%%%%%%%%%%%%%%%%%%%%%%%%%%%%%%%%%%%%%%%%%%%%%%%%%%%%%%%
Let $c:X\onto G$ be a surjective generalized subgroup of a finite simple group
$G$.
Then any non-trivial homomorphism $f:X\ra G$ can be expressed as a composition of
$c:X\onto G$ and some automorphism $G\ra G$.
\end{lemma}
\begin{proof}
Let $K_f=\text{Ker}(f)$  and $K_c=\text{Ker}(c)$.
Consider the following commutative diagram:
\[\xymatrix{
K_f\cap K_c\rmono \dmono & K_c\dmono \\
K_f\rmono \drto^{g}& X\rto^f \depi^{c} & G\\
& G
}\]
The image of $g$ is a normal subgroup of $G$. Since $G$ is simple, there are two possibilities.
Either $g$ is a surjection or it is the trivial homomorphism.

Assume that  $g$ is trivial. In this case $K_f$ is a subgroup of $K_c$ and we have the following commutative diagram:
\[\xymatrix{
X\rto^f\ar@{->>}[d]_{c}\ar@{->>}[dr] & G\\
G & X/K_f\ar@{^(->}[u]\ar@{->>}[l]
}\]
Finiteness of $G$ implies that the surjection  $X/K_f\onto G$ and the injection $X/K_f\hookrightarrow  G$ in the above diagram have to be isomorphisms. We can then use the commutativity of this  diagram to conclude that
$f$ can be expressed as a composition of $c:X\onto G$ and some automorphism $G\ra G$.

We will show that under the assumption that $f$ is non-trivial, the homomorphism $g$
can not be surjective.
Assume to the contrary that  $g$ is surjective.  In this case we can use~\ref{prop propgsub}(5) to get
equality $K_f=X$.  Consequently the surjectivity of $g$
implies the triviality of $f$.
\end{proof}

%%%%%%%%%%%%%%%%%%%%%%%%%%%%%%%%%%%%%%%%%%%%%%%%%%%%%%%%%%%%
%8.4
\begin{cor}
%%%%%%%%%%%%%%%%%%%%%%%%%%%%%%%%%%%%%%%%%%%%%%%%%%%%%%%%%%%%%%%
Let $G$ be a finite simple group and $X\in\text{\rm Idem}(G)$. Then any non-trivial
homomorphism $f:X\ra G$ is a cellular cover of $G$.
\end{cor}
\begin{proof}
Let $X\in\text{\rm Idem}(G)$. According to~\ref{prop identifyingidemp},  there is a homomorphism $c:X\ra G$ which is a cellular cover. The image of $c$ is a normal
subgroup of $G$. It is then either  the trivial group or the whole group $G$. In the first case
$X$ is the trivial group and there is no non-trivial homomorphisms from $X$ to $G$. In the second case
we can use~\ref{lem basicsimplgr}, to conclude that $f:X\ra G$
can be expressed as a composition of $c:X\onto G$ and some automorphism $G\ra G$.
It is then clear that $\text{Hom}(X,f):\text{Hom}(X,X)\ra \text{Hom}(X,G)$ is a bijection
and  therefore $f$ is a cellular cover.
\end{proof}

\begin{proof}[Proof of~\ref{thm surcovfs}]
Let  $c:X\onto G$ be a surjective generalized subgroup whose differential
$\mu(c):H_{2}(G)\onto K_c$ represents an element in
$\text{\rm InvQuot}(H_{2}(G))$.
We will use~\ref{cor gscov} to prove the theorem.
According to~\ref{lem basicsimplgr}, any non-trivial homomorphism $f:X\ra G$ is a composition of
$c:X\onto G$ and an automorphism $h:G\ra G$. Consequently
$H_{2}(f)=H_{2}(h)H_{2}(c)$.
As the image of $H_{2}(c)$ is fixed by the action of $\text{Out}(G)$ on $\gensub{H_{2}(G)}$,
we have:
\[\text{image}(H_{2}(f))=\text{image}\left(H_{2}(h)H_{2}(c)\right)=
\text{image}(H_{2}(c))\]
By~\ref{cor gscov}, $c:X\onto G$ is then a cellular cover.
\end{proof}

%%%%%%%%%%%%%%%%%%%%%%%%%%%%%%%%%%%%%%%%%%%%%%%%%%%%%%%%%%%
%%%%%%%%%%%%%%%%%%%%%%%%%%%%%%%%%%%%%%%%%%%%%%%%%%%%%%%%%%
%%%%%%%%%%%%%%%%%%%%%%%%%%%%%%%%%%%%%%%%%%%%%%%%%%%%%%%%%%
%section9
\section{Iterated generalized subgroups and cellular covers}
%%%%%%%%%%%%%%%%%%%%%%%%%%%%%%%%%%%%%%%%%%%%%%%%%%%%%%%%%
%%%%%%%%%%%%%%%%%%%%%%%%%%%%%%%%%%%%%%%%%%%%%%%%%%%%%%%%%%
%%%%%%%%%%%%%%%%%%%%%%%%%%%%%%%%%%%%%%%%%%%%%%%%%%%%%%%%%%
Consider   the group $\text{PSL}_2(q)$ where $q$ is a power of an odd prime and distinct from $3$ and $9$. The initial
cellular cover of $\text{PSL}_2(q)$ is represented by
the universal central extension which is a surjection
$e:\text{SL}_2(q)\onto \text{PSL}_2(q)$ whose kernel is the center of $\text{SL}_2(q)$. This center is
isomorphic to ${\mathbf Z}/2$ and it is the only subgroup of $\text{SL}_2(q)$ isomorphic to  ${\mathbf Z}/2$.
Consequently   the inclusion ${\mathbf Z}/2\subset \text{SL}_2(q)$ is also a cellular cover.
The composition of these two cellular covers ${\mathbf Z}/2\subset \text{SL}_2(q)\onto \text{PSL}_2(q)$ is the trivial
homomorphism which is not a generalized subgroup. Thus in general neither  the composition of two cellular covers is a cellular cover nor is the composition of two
generalized subgroups a generalized subgroup. The aim of this section is to discuss
the possible homomorphisms  that are obtained as  compositions of several  cellular covers and of repeated generalized subgroups.

%%%%%%%%%%%%%%%%%%%%%%%%%%%%%%%%%%%%%%%%%%%%%%%%%%%%
%9.1
 \begin{Def}
%%%%%%%%%%%%%%%%%%%%%%%%%%%%%%%%%%%%%%%%%%%%%%%%%%%
Let  $G$ be a group and $n$ a positive integer.
\begin{enumerate}
\item A homomorphism $a:X\ra G$ is defined to be an $n$-{\bf iterated generalized subgroup} of $G$ if $a$ can be expressed as a composition:
\[\xymatrix{ X\ar@{=}[r]\ar@/ _ 15pt/[rrrr]|{a} &X_n\rto^{a_n} & \cdots\rto & X_1\rto^{a_1} & G
}\]
where all  $a_i$'s are generalized subgroups, i.e., $a_1:X_1\ra G$ is a generalized subgroup of $G$, $a_2:X_2\ra X_1$ is a generalized subgroup of $X_1$,  $a_3:X_3\ra X_2$ is a generalized subgroup of $X_2$,
etc.
Two such $n$-iterated generalized subgroups   $a:X\ra G$ and $b:Y\ra G$ are defined to be equivalent,  if
 there is
an isomorphism $h:X\ra Y$ such that  $bh=a$.
\item
Analogously  $a:X\ra G$ is defined to be  an $n$-{\bf iterated surjective generalized subgroup}, or
$n$-{\bf iterated cellular cover}, or  $n$-{\bf iterated surjective cellular cover} of $G$ if $a$ can be expressed as a
 composition of respectively $n$
surjective generalized subgroups, or $n$  cellular covers, or $n$  surjective cellular covers. Two such homomorphisms $a$ and $b$ are equivalent, if there is
an isomorphism $h$ such that  $bh=a$.
\end{enumerate}
\end{Def}

Iterated cellular covers and iterated generalized subgroups have the following two features in common with the ordinary cellular covers and generalized subgroups:

%%%%%%%%%%%%%%%%%%%%%%%%%%%%%%%%%%%%%%%%%%%%%%%%%%%%%%%%%%%%%%%%
%9.2
\begin{prop}\label{prop twofetit}
%%%%%%%%%%%%%%%%%%%%%%%%%%%%%%%%%%%%%%%%%%%%%%%%%%%%%%%%%%%%%%%%%%%%%
Let $G$ be a group.
\begin{enumerate}
\item If $a:X\ra G$ is an $n$-iterated generalized subgroup of $G$, then  its kernel is
central in $X$.
\item If $c:X\ra G$ is an $n$-iterated cellular cover of $G$, then its image is a fully invariant subgroup of $G$.
\end{enumerate}
\end{prop}
\begin{proof}
\noindent (1):\quad
First notice that if $a: X\to Y$ is a generalized subgroup, then $a^{-1}(Z(Y))\subset Z(X)$.
Indeed, for $x\in a^{-1}(Z(Y))$ the map $a\circ c_x=a$, where $c_x$ is conjugation by $x$,
hence $c_x={\rm id},$ so $x\in Z(X)$.  Now by induction on $n$ it is easily seen
that the kernel of an $n$-iterated generalized subgroup $a: X\ra G$ is central in $X$.
\smallskip

\noindent (2):\quad    Assume $c:X\ra G$ is a composition of  $n$ cellular covers
$X_n\xrightarrow{c_n}  \cdots\ra  X_1\xrightarrow{c_1}  G$. For any homomorphism $h:G\ra G$, using the fact that $c_i$'s are cellular covers, there is a unique sequence of homomorphisms $\{h_i:X_i\ra X_i\}_{1\leq i\leq n}$ for which the following diagram commutes:
\[\xymatrix{ X\ar@{=}[r]\ar@/^ 20pt/[rrrr]|{c} &X_n\ar@{->>}[r]^(.45){c_n}\dto_{h_n} & \cdots\rto & X_1\ar@{->>}[r]^(.42){c_1}\dto^{h_1} & G\dto^{h}\\
X\ar@{=}[r]\ar@/_ 15pt/[rrrr]|{c} &X_n\ar@{->>}[r]^(.45){c_n} & \cdots\rto & X_1\ar@{->>}[r]^{c_1} & G
}
\]
Commutativity of this diagrams implies that $h$ maps the image of $c$ into itself.
\end{proof}

%%%%%%%%%%%%%%%%%%%%%%%%%%%%%%%%%%%%%%%%%%%%%%%%%%%%%%%%%%%%%%%%%%%%%%%%%%%%%%
%9.3
\begin{Def}\label{def iteratedcollections}
%%%%%%%%%%%%%%%%%%%%%%%%%%%%%%%%%%%%%%%%%%%%%%%%%%%%%%%%%%%%%%%%%%%%%%%%%%%%%%
Let $G$ be a group and $n$ a positive integer.
\begin{enumerate}
\item $\itgensub{G}{n}$ denotes the collection of equivalence classes of $n$-iterated generalized subgroups of $G$.
\item  $\itsurgensub{G}{n}$ denotes the collection of equivalence classes of $n$-iterated surjective generalized subgroups of $G$.
\item  $\itcov{G}{n}$ denotes the collection of equivalence classes of $n$-iterated cellular covers  of $G$.
\item $\itsurcov{G}{n}$ denotes the collection of equivalence classes of $n$-iterated surjective cellular covers  of $G$.
\end{enumerate}
\end{Def}
According to this definition $\gensub{G}=\itgensub{G}{1}$,
$\surgensub{G}=\itsurgensub{G}{1}$, $\cov{G}=\itcov{G}{1}$,
and $\surcov{G}=\itsurcov{G}{1}$.

The identity homomorphisms are  cellular covers and  hence the following inclusions hold:
\[\xymatrix@R-10pt{
\itsurcov{G}{n}\ar@{^(->}[rr] \ar@{^(->}[dd] \ar@{^(->}[dr]& & \itsurgensub{G}{n}\ar@{^(->}[dd]|\hole \ar@{^(->}[dr]\\
& \itsurcov{G}{n+1}\ar@{^(->}[rr] \ar@{^(->}[dd] & &  \itsurgensub{G}{n+1}\ar@{^(->}[dd]\\
\itcov{G}{n}\ar@{^(->}[rr]|\hole  \ar@{^(->}[dr]& & \itgensub{G}{n} \ar@{^(->}[dr]\\
 & \itcov{G}{n+1}\ar@{^(->}[rr]& & \itgensub{G}{n+1}\\
}\]
By summing up these  inclusions we can  extend Definition~\ref{def iteratedcollections} to:
\begin{Def} Let $G$ be a group. Define:
\[\xymatrix@R-10pt{\itsurcov{G}{\infty}:=\bigcup_{n\geq 1}\itsurcov{G}{n}
&  \itsurgensub{G}{\infty}:=\bigcup_{n\geq 1}  \itsurgensub{G}{n},\\
\itcov{G}{\infty}:=\bigcup_{n\geq 1} \itcov{G}{n}&
 \itgensub{G}{\infty}:=\bigcup_{n\geq 1}  \itgensub{G}{n}}\]
\end{Def}

We are interested in the above collections primary in the case when $G$ is finite.
Recall that  a finite group $G$ has a composition series $1= G_0\subset \cdots\subset G_l=G$ and any such series has the same length $l$ which is called   the composition length of $G$.

%%%%%%%%%%%%%%%%%%%%%%%%%%%%%%%%%%%%%%%%%%%%%%%%%%%%%%%%%%%%%%%%%%%%%%%%%%%%%%%%%%%%%%%%%%%%%%%%%%
%9.5
\begin{thm} \label{thm iteratedgsubfinite}
%%%%%%%%%%%%%%%%%%%%%%%%%%%%%%%%%%%%%%%%%%%%%%%%%%%%%%%%%%%%%%%%%%%%%%%%%%%%%%%%%%%%%
Let  $G$ be a finite group and $l$ its  composition length.
Then  $ \itgensub{G}{l+1}= \itgensub{G}{\infty}$.
\end{thm}
To prove the theorem we need:

%%%%%%%%%%%%%%%%%%%%%%%%%%%%%%%%%%%%%%%%%%%%%%%%%%%%%%%%%%%%%%%%%%%%%%%%%%%%%%
%9.6
\begin{lemma}\label{lem keycomgsub}
%%%%%%%%%%%%%%%%%%%%%%%%%%%%%%%%%%%%%%%%%%%%%%%%%%%%%%%%%%%%%%%%%%%%%%%%%%%%%%%%%
Let  $G$ be  finite and  $X_2\xrightarrow{a_2}X_1\xrightarrow{a_1} G$ be  generalized subgroups.
\begin{enumerate}
\item If any prime that divides the order of $H_1(X_2)$ divides also the  order of $H_1(X_1)$, then the composition $a_1 a_2:X_2\ra G$ is a generalized subgroup.
\item If  $a_2$ is surjective,  then the composition $a_1 a_2:X_2\ra G$ is a generalized subgroup.
\end{enumerate}
\end{lemma}
\begin{proof}
\noindent (1):\quad
According to~\ref{prop twofetit}, the kernel $K_{a_1a_2}$ of the composition  $a_1a_2:X_2\ra  G$ is central in $X_2$. Thus to show that
$a_1a_2$ is a generalized subgroup, we need to prove $\text{Hom}(X_2, K_{a_1a_2})=0$
(see~\ref{prop charcCmono}).  The kernels
$K_{a_1}$ and $K_{a_2}$ of $a_1$ and $a_2$ and the kernel $K_{a_1a_2}$ are abelian groups and they fit
into an exact sequence
$ K_{a_2}\subset K_{a_1a_2}\onto K_{a_1}$. Since  $\text{Hom}(X_2, K_{a_2})=0$, we get an inclusion:
\[\text{Hom}(X_2, K_{a_1a_2})\subset \text{Hom}(X_2,  K_{a_1})=\text{Hom}(H_1(X_2),  K_{a_1}).\]
Since $a_1$ is a generalized subgroup $\text{Hom}\left(H_1(X_1),  K_{a_1}\right)=
\text{Hom}(X_1,  K_{a_1})=0$. This means that the primes that divide the order of $K_{a_1}$ do not divide then order
of $H_1(X_1)$.  By the assumption  the primes that divide
the order of $K_{a_1}$ can  not divide the order of $H_1(X_2)$  either. Consequently
$\text{Hom}\left(H_1(X_2),  K_{a_1}\right)=0$.
\medskip

\noindent (2):\quad
If $a_2:X_2\ra X_1$ is surjective, then  $H_1(X_2)$ and $H_1(X_1)$ are isomorphic
(see~\ref{prop propgsub}(1)). Statement (2) follows thus  from (1).
\end{proof}

%%%%%%%%%%%%%%%%%%%%%%%%%%%%%%%%%%%%%%%%%%%%%%%%%%%%%%%%%%%%%%%%%%%%%%%%%%%%%%%
\begin{proof}[Proof of~\ref{thm iteratedgsubfinite}]
%%%%%%%%%%%%%%%%%%%%%%%%%%%%%%%%%%%%%%%%%%%%%%%%%%%%%%%%%%%%%%%%%%%%%%%%%%%%%%%%%%%%%%%%%%%%%%%%
Assume that $a:X\ra G$ can be expressed  as a composition of generalized subgroups
$X_n\xrightarrow{a_n}\cdots\ra X_1\xrightarrow{a_1} G$.  By performing compositions  if necessary, we can assume that {\bf none} of the adjacent composition
$a_ia_{i+1}$ is a generalized subgroup. To prove the theorem we need to show that
$n\leq l+1$. To do that it is enough to prove that the image of the composition $a_{1}\cdots a_ia_{i+1}:X_{i+1}\ra G$ is a {\bf proper} subgroup
of the image of $a_{1}\cdots a_i:X_i\ra G$ for any $1\leq i\leq n-1$.
This is because these images will lead to a proper normal series of length $n-1$
in $G$.  Since the composition length of $G$ is $l$, we must have $n\le l+1$.

Assume that this is not the case. Let $i$ be
an index for which $b:=a_{1}\cdots a_i$ and $ba_{i+1}$ have the same image in $G$.
Let $I_{a_{i+1}}=a_{i+1}(X_{i+1})$ and let $K_b:=\Ker b$.  By~\ref{prop twofetit}(1),
$\Ker b$ is central in $X_i$ and by hypothesis $X_i=I_{a_{i+1}}K_b$.  Hence
$[X_i, X_i]=[I_{a_{i+1}}, I_{a_{i+1}}]$ and it follows that $H_1(I_{a_{i+1}})$
is a subgroup of $H_1(X_i)$.  But according to~\ref{prop propgsub}(1),
$H_1(X_{i+1})\cong H_1(I_{a_{i+1}})$.  This together with~\ref{lem keycomgsub}(1)
shows that $a_ia_{i+1}$ is a generalized subgroup, a contradiction.
\end{proof}

%%%%%%%%%%%%%%%%%%%%%%%%%%%%%%%%%%%%%%%%%%%%%%%%%%%%%%%%%%%%%%%%%%%%%%%%%%%%%%%%%%%%%%%
%9.7
\begin{cor} \label{cor basicconsmth}
%%%%%%%%%%%%%%%%%%%%%%%%%%%%%%%%%%%%%%%%%%%%%%%%%%%%%%%%%%%%%%%%%%%%%%%%%%%%%%%%%
Let $G$ be a finite group.
\begin{enumerate}
\item  $\itgensub{G}{\infty}$, $\itcov{G}{\infty}$, $\itsurgensub{G}{\infty}$,
and  $\itsurcov{G}{\infty}$ are finite sets.
\item There is a positive integer $N$ for which $\itcov{G}{N}=\itcov{G}{\infty}$.
\item For any  positive integer $n$, $\surgensub{G} =\itsurgensub{G}{n}=\itsurgensub{G}{\infty}$.
\end{enumerate}
\end{cor}
\begin{proof}
\noindent (1):\quad
A direct consequence of~\ref{cor classificationgsub} and
the fact that
a generalized subgroup of a finite group is finite (see~\ref{prop presofsolnilfin}(2))
is finiteness of $\itgensub{G}{n}$ for any $n$.  Finiteness of $\itgensub{G}{\infty}$
is then a consequence of~\ref{thm iteratedgsubfinite}.
\medskip

\noindent (2):\quad
Since $\itcov{G}{\infty}$ is finite and it is a sum of an increasing sequence of sets
$\itcov{G}{1}\subset \itcov{G}{2}\subset\cdots$, we mast have
that, for some
$N$, $\itcov{G}{N}=\itcov{G}{\infty}$.
\medskip

\noindent (3):\quad
According to~\ref{lem keycomgsub}(2) a composition of surjective generalized subgroups is a
generalized subgroup, which is obviously surjective.
\end{proof}

For a finite group $G$, we have an inclusion $\itsurcov{G}{\infty}\subset \surgensub{G}$
(see~\ref{cor basicconsmth}(3)).
 To understand  some constrains on elements of
$ \surgensub{G}$ which  lie in $\itsurcov{G}{\infty}$, let us recall  that,
according to~\ref{thm classificationsurgsub},
we have   the following bijections:
\[\xymatrix@R-20pt{
\surgensub{G}\rto  & \text{Quot}(H_{2\setminus 1}(G)) \rto & \gensub{(H_{2\setminus 1}(G)} \\
{(c:X\onto G)}\ar@{|->}[r]&  {(\mu(c):H_{2\setminus 1}(G)\onto K_c)}\ar@{|->}[r] &H_{2\setminus 1}(X).
 }\]

%%%%%%%%%%%%%%%%%%%%%%%%%%%%%%%%%%%%%%%%%%%%%%%%%%%%%%%%%%%%%%%%%%%%%%%
%9.8
\begin{prop}\label{prop invarianceintersurcov}
%%%%%%%%%%%%%%%%%%%%%%%%%%%%%%%%%%%%%%%%%%%%%%%%%%%%%%%%%%%%%%%%%%%%%%
Let $G$ be a finite group. If $c:X\onto G$  is an   $n$-iterated surjective cellular cover, then its differential $\mu(c):
H_{2\setminus 1}(G)\onto K_c$
represents an element in  $\invquot{H_{2\setminus 1}(G)}$.
\end{prop}
\begin{proof}
The argument is exactly the same as in the proof of~\ref{prop diffcovfix}.
Assume that  $c:X\onto G$ can be expressed as a composition of $n$-surjective cellular covers:
$X_n\stackrel{c_n}{\onto}\cdots\onto X_1\stackrel{c_1}{\onto} G$.
%\[\xymatrix{ X\ar@{=}[r]\ar@/ _ 20pt/[rrrr]|{c} &X_n\ar@{->>}[r]^(.45){c_n} & \cdots\rto & X_1\ar@{->>}[r]^{c_1} & G
%}\]
For any isomorphism $h:G\ra G$, using the fact that $c_i$'s are cellular covers, there is a unique sequence of isomorphisms $\{h_i:X_i\ra X_i\}_{1\leq i\leq n}$ for which the following diagram commutes:
\[\xymatrix{ X\ar@{=}[r]\ar@/^ 20pt/[rrrr]|{c} &X_n\ar@{->>}[r]^(.45){c_n}\dto_{h_n} & \cdots\rto & X_1\ar@{->>}[r]^(.42){c_1}\dto^{h_1} & G\dto^{h}\\
X\ar@{=}[r]\ar@/_ 15pt/[rrrr]|{c} &X_n\ar@{->>}[r]^(.45){c_n} & \cdots\rto & X_1\ar@{->>}[r]^{c_1} & G
}
\]
By the naturality of the differential we get an induced commutative  diagram:
\[\xymatrix{
H_{2\setminus 1}(G)\ar@{->>}[r]^(.55){\mu(c)}\dto_{H_{2\setminus 1}(h)} & K_c\dto^{h_n}\\
H_{2\setminus 1}(G)\ar@{->>}[r]^(.55){\mu(c)} & K_c
}\]
Commutativity of this diagram implies that $H_{2\setminus 1}(h)$ maps the kernel of
$\mu(c)$ onto itself. Since  $h$ was an arbitrary automorphism of $G$, the kernel of
$\mu(c)$ belongs to $\text{Inv}\gensub{H_{2\setminus 1}(G)}$ and consequently
$\mu(c):
H_{2\setminus 1}(G)\onto K_c$
represents an element in  $\invquot{H_{2\setminus 1}(G)}$.
\end{proof}

%%%%%%%%%%%%%%%%%%%%%%%%%%%%%%%%%%%%%%%%%%%%%%%%%%%%%%%%%%%%%
%9.9
\begin{cor}\label{cor acycsimpsurcov}
%%%%%%%%%%%%%%%%%%%%%%%%%%%%%%%%%%%%%%%%%%%%%%%%%%%%%%%%%%%%%%%%
Let $G$ be a finite group.
\begin{enumerate}
\item  If $H_{2\setminus 1}(G)$ is cyclic, then, for any positive  integer $n$:
\[\surcov{G}=\itsurcov{G}{n}=\itsurcov{G}{\infty}=\surgensub{G}.\]
\item  If $G$ is  simple, then, for any positive integer $n$:
\[\surcov{G}=
\itsurcov{G}{n}=\itsurcov{G}{\infty}.\]
\end{enumerate}
\end{cor}
\begin{proof}
\noindent (1):\quad
This is a consequence of the inclusion $\itsurcov{G}{\infty}\subset \surgensub{G}$,
which
follows from~\ref{cor basicconsmth}(3), and  the equality   $\surcov{G}=\surgensub{G}$ (see~\ref{cor sumcyclicH2min1}(2)).
\smallskip

\noindent (2):\quad
By~\ref{prop invarianceintersurcov}, the differential $\mu:\itsurcov{G}{\infty}\xrightarrow{\mu}
\text{Inv}\gensub{H_{2\setminus 1}(G)}$ is an injection.
Since its restriction $\surcov{G}\subset\itsurcov{G}{\infty}\xrightarrow{\mu} \text{Inv}\gensub{H_{2\setminus 1}(G)}$ is a bijection
(see~\ref{thm surcovfs}) the statement follows.
\end{proof}

We finish this section with:
%%%%%%%%%%%%%%%%%%%%%%%%%%%%%%%%%%%%%%%%%%%%%%%%%%%%%%%%%%%%%%%%%
%9.10
\begin{prop}\label{prop itercovsimple}\hspace{1mm}
%%%%%%%%%%%%%%%%%%%%%%%%%%%%%%%%%%%%%%%%%%%%%%%%%%%%%%%%%%%%%%%%%
\begin{enumerate}
\item If $G$ is finitely generated nilpotent, then, for any positive integer $n$,
$\gensub{G}=\itgensub{G}{n}=\itgensub{G}{\infty}$.
\item If $A$ is finite abelian, then, for any positive integer $n$,
$\cov{A}=\itcov{A}{n}=\itcov{A}{\infty}$.
\item If $G$  is finite and simple, then $\itcov{G}{2}=\itcov{G}{\infty}$.
\end{enumerate}
\end{prop}
\begin{proof}
\noindent (1):\quad   A generalized subgroup of a finitely generated nilpotent group is an injection
(see~\ref{prop presofsolnilfin}(3)).
The statement follows from  the fact that the  composition of injections is an injection.
\smallskip

\noindent (2):\quad   If $A$ is finite abelian then, according to~\ref{prop covfabelian}, the cellular covers of $A$ are given by the $k$-torsion subgroups of $A$. The $m$-torsion subgroup of the $k$-torsion subgroup of $A$ is simply
the $\text{lcm}(k,m)$-torsion subgroup of $A$ and hence it is also a cellular cover of $A$.  Composition of two such cellular covers is then a cellular cover which shows (2).
\smallskip

\noindent (3):\quad Let $c:X\ra G$ belong to  $\itcov{G}{n}$.
We will show by induction on $n$ that $c\in \itcov{G}{2}$. If $n=2$, there is nothing to prove. Let   $n>2$. Assume that  $c$ can be expressed as a composition of  cellular covers  $X_n\xrightarrow{c_n}  \cdots\ra  X_1\xrightarrow{c_1}  G$.  Since $G$ is simple, there are two possibilities: either $X_1=0$ or $c_1$ is a
surjection. In the first case any cellular cover of $X_1$ is the trivial group too. The same holds for all $X_i$'s. Thus   $X=0$ and $c$ is the trivial homomorphism, in particular $c\in \cov{G}$.

Assume $c_1$ is a surjection.
The image of $c_1c_2:X_2\ra G$ is a fully invariant subgroup in $G$
(see~\ref{prop twofetit}(2)). Thus it is either the whole group $G$ or it is the trivial group. In the first case,
according to~\ref{prop propgsub}(5), $c_2$ is a surjection. We can then use~\ref{cor acycsimpsurcov}(2) to conclude that
the composition $c_1c_2:X_2\ra G$ is a surjective cellular cover of $G$. In this case  $c$ can be expressed as a composition of $n-1$ cellular covers.  By the inductive assumption $c\in  \itcov{G}{2}$.

The last possibility is that the image of $c_1c_2$ is the trivial group which means that the image of $c_2$
sits in the kernel $K_{c_1}$ of $c_1$.
We will prove that  the composition $c_2c_3\cdots c_n:X\ra X_1$ is a cellular cover.
In this way we can express $c$ as a composition of two cellular covers proving  the statement.

The kernel  $K_{c_1}$ is a finite abelian group.  Since $c_2:X_2\ra K_{c_1}$
is also a cellular cover, $c_2$ must be an injection and consequently $X_2$ is an abelian group. We can then use  (2) to get that the composition  $c_3\cdots c_n:X\ra  X_2$ is a cellular cover, which in particular means that it is an  injection.
Let $f:X\ra X_1$ be an arbitrary homomorphism. As  $X$ is (isomorphic to) the $k$-torsion subgroup (for some $k$)
of the finite abelian group $X_2$, there is a surjection
$\pi:X_2\onto X$. Because  $c_2:X_2\subset X_1$ is a cellular cover, the image of the composition $\pi f:X_2\ra X_1$ sits in the image of $c_2$.
Thus $f$ has to map $X$ into the image of $c_2$. The fact that $c_3\cdots c_n:X\ra  X_2$  is a cellular cover implies that $f$
has to map $X$   into the image of $c_3\cdots c_n$. This means that  $f$ factors through $c_2c_3\cdots c_n:X\ra X_1$, proving that the later homomorphism is a cellular cover.
\end{proof}

The statement (3) in~\ref{prop itercovsimple} can not be made stronger. For example, in the case of
the group $\text{PSL}_2(q)$ where $q$ is a power of an odd prime different from $3$ and $9$, there is a proper inclusion $\cov{\text{PSL}_2(q)}\subsetneq \itcov{\text{PSL}_2(q)}{2}=
\itcov{\text{PSL}_2(q)}{\infty}$.

%%%%%%%%%%%%%%%%%%%%%%%%%%%%%%%%%%%%%%%%%%%%%%%%%%%%%%%%%%%%%%%%
%%%%%%%%%%%%%%%%%%%%%%%%%%%%%%%%%%%%%%%%%%%%%%%%%%%%%%%%%%%%%%%%%
%%%%%%%%%%%%%%%%%%%%%%%%%%%%%%%%%%%%%%%%%%%%%%%%%%%%%%%%%%%%%%%%%
%section10
\section{Iterating idempotent functors}
%%%%%%%%%%%%%%%%%%%%%%%%%%%%%%%%%%%%%%%%%%%%%%%%%%%%%%%%%%%%%%%%%
%%%%%%%%%%%%%%%%%%%%%%%%%%%%%%%%%%%%%%%%%%%%%%%%%%%%%%%%%%%%%%
%%%%%%%%%%%%%%%%%%%%%%%%%%%%%%%%%%%%%%%%%%%%%%%%%%%%%%%%%%
Recall that according to~\ref{prop identifyingidemp}
the function {\bf F} that assigns to an equivalence class of a cellular cover, represented by $c:X\ra G$, the group $X$ induces a {\bf bijection} between $\cov{G}$  and
$\text{Idem}(G)$.   This means that if $X$ represents an element in $\text{Idem}(G)$, then up to an isomorphism of $X$, there is a unique homomorphism $c:X\ra G$ which is a cellular cover of $G$.

For any positive integer $n$, the function {\bf F} induces a surjection (which we also call {\bf F})
from  $\itcov{G}{n}$  onto
$\text{Idem}^{n}(G)$.   This  surjectivity and~\ref{cor basicconsmth}(1) gives:
\begin{cor}\label{cor ideminffin}
If $G$   is a finite, then $\text{\rm  Idem}^{\infty}(G)$ is a finite set.
\end{cor}

In general we do not know if  the   function {\bf F} induces a {\bf bijection}  between $\itcov{G}{n}$
and $\text{Idem}^{n}(G)$. We do not know if, for a group $X$, that represents an element in
$ \text{Idem}^{n}(G)$, there is a unique, up to an isomorphism of $X$, homomorphism $f:X\ra G$ which is a composition of a sequence of $n$ cellular covers.  We can show however that this   is true  for finite simple groups:

\begin{prop}\label{prop idemiterfinsimple}
Let $G$ be a finite simple group. Then the function that assigns to an element in
$\itcov{G}{2}$, represented by $f:X\ra G$, the element in $\text{\rm  Idem}^{\infty}(G)$, represented by $X$, is a bijection.
\end{prop}
\begin{proof}
If $G$ is abelian, then the proposition follows from~\ref{prop identifyingidemp} since in this case
$\cov{G}=\itcov{G}{\infty}$
(see~\ref{prop itercovsimple}(2)).

Assume $G$   is not abelian.
According to~\ref{prop itercovsimple}(3),  $\itcov{G}{2}=\itcov{G}{\infty}$. Thus the
  function {\bf F}
 between $\itcov{G}{2}$ and $\text{\rm  Idem}^{\infty}(G)$ is  surjective.

 Let $X\in \text{\rm  Idem}^{\infty}(G)$  be a  non-trivial group.
There are two possibilities: $X$ is abelian or not. In the first case we claim that the trivial homomorphism from   $X$ to $G$
is the only homomorphism that represents an element in  $\itcov{G}{2}$.  Assume that this is not the case.
Then there are two cellular covers $c_2:X\ra X_1$ and $c_1:X_1\ra G$ whose composition $c_1c_2:X\ra G$ is non trivial.
As the image of such a composition is a fully invariant subgroup of $G$ (see~\ref{prop twofetit}), $c_1c_2$  has to be  a surjection, which
can only happen in the case $G$ is abelian.

Assume that $X$ is not abelian. If the composition $c_1c_2$ were trivial, then $c_2:X\ra K_{c_1}$
would be a cellular cover
of the kernel $K_{c_1}$ which would require $X$ to be abelian. The composition $c_1c_2$ is therefore non-trivial and hence it has to be a surjection.
The homomorphism $c_1$ is then also a surjection.  By~\ref{prop propgsub}(5) we can  conclude that $c_2$ is a surjection too.  The composition $c_1c_2$ is then a cellular
cover (see~\ref{cor acycsimpsurcov}(2)), and hence the homomorphism $c_1c_2$ is unique, up to an automorphism of $X$.
\end{proof}

%%%%%%%%%%%%%%%%%%%%%%%%%%%%%%%%%%%%%%%%%%%%%%%%%%%%%%%%%%%%%%%%%%%
%%%%%%%%%%%%%%%%%%%%%%%%%%%%%%%%%%%%%%%%%%%%%%%%%%%%%%%%%%%%%%%%%%%
%%%%%%%%%%%%%%%%%%%%%%%%%%%%%%%%%%%%%%%%%%%%%%%%%%%%%%%%%%%%%%%%%%
%section11
\section{Explicit examples}
%%%%%%%%%%%%%%%%%%%%%%%%%%%%%%%%%%%%%%%%%%%%%%%%%%%%%%%%%%%%%%%%%%%
%%%%%%%%%%%%%%%%%%%%%%%%%%%%%%%%%%%%%%%%%%%%%%%%%%%%%%%%%%%%%%%%%%
%%%%%%%%%%%%%%%%%%%%%%%%%%%%%%%%%%%%%%%%%%%%%%%%%%%%%%%%%%%%%%
In this section we will illustrate Theorem~\ref{thm surcovfs} and its Corollary~\ref{cor covsimple}.
We let $G$ denote a  finite simple group. We use the symbol $\text{exp}(H_2(G))$ to denote the exponent of $H_2(G)$ and
$\sigma_{0}(G)$ the number of different positive divisors of $\text{exp}(H_2(G))$.
Recall that the exponent of a finite abelian group $A$ is the least positive integer $k$ for which $kA=0$.
Let
 $e_G:E\onto G$ be  the  universal central extension of $G$.
The center of $E$ is isomorphic to $H_2(G)$.

According to~\ref{cor covsimple}, $\text{Idem}(G)$ is in bijection with the set
$\{0\}\coprod\invquot{H_{2}(G)}$. Explicitly,  the element $0$ corresponds to the trivial group in   $\text{Idem}(G)$. Any non-trivial
 element in $\text{Idem}(G)$ is the quotient of $E$ by an $\text{Out}(G)$-invariant subgroup of its center   $H_2(G)$. A basic example of such a subgroup
is given by the $k$-torsion subgroup for some $k$ dividing the exponent  $\text{exp}(H_2(G))$ of  $H_2(G)$ (see~\ref{cor cyclictrivialaction}(1)).
The number of such basic invariant subgroups of $H_2(G)$ is therefore given by $\sigma_{0}(G)$.
Thus the set  $\text{Idem}(G)$ contains at least $\sigma_{0}( H_2(G))+1$ elements.
The question is if there are  any other invariant subgroups of $H_2(G)$?  For example in the case
$H_2(G)$ is cyclic,  since the action of $\text{Out}(G)$ on $\gensub{H_2(G)}$ is trivial  (see~\ref{cor cyclictrivialaction}),
all the subgroups of $H_2(G)$ are invariant. In this case the set $\text{Idem}(G)$ has exactly $\sigma_{0}( H_2(G))+1$ elements.
{\em It turns  out that this happens for almost all simple groups.} The only exceptions are the  groups  $D_n(q)$ $(n\geq 3)$ for $q$ odd and $n$ even. In this case the Schur multiplier is  ${\mathbf Z}/2\oplus {\mathbf Z}/2$  and hence  its exponent is $2$ and
consequently  $\sigma_{0}(D_n(q))=2$. However, the number of invariant subgroups in the Schur multiplier turns out to be  $3$ and hence
 $\text{Idem}(D_n(q))$ has 4 elements.

%%%%%%%%%%%%%%%%%%%%%%%%%%%%%%%%%%%%%%%%%%%%%%%%%%%%%%%%%%%%%%%%%%%%%%%%
%
 \begin{prop}\label{prop listoffsg}
%%%%%%%%%%%%%%%%%%%%%%%%%%%%%%%%%%%%%%%%%%%%%%%%
 The following table  lists  the size of $\text{\rm Idem}(G)$ (fifth column) for all finite simple groups $G$.
In the first column the boxed entries contain the names of the groups with restrictions on relevant indices that are required for the groups to be simple or to avoid some repetition.  The notation is taken from~\cite{GLS1}.   The constrains below the boxes
 distinguish between different Schur multipliers.  Schur multipliers are the content of the second column. The third column contains  $\text{\rm exp}(H_2(G)$  and the forth $\sigma_0(G)$.
 In the first column, we write
 $\bullet$\fbox{$G$}$\bullet$ to denote these groups $G$ for which the Schur multiplier is not cyclic and yet
 $\text{\rm Idem}(G)$ has $\sigma_{0}( H_2(G))+1$ elements.
 We use $\bullet$$\bullet$\fbox{$G$}$\bullet$$\bullet$ to  denote the cases for which
 $\text{\rm Idem}(G)$ has  more than $\sigma_{0}( H_2(G))+1$  elements.
\renewcommand{\arraystretch}{1.2}
\begin{center}
\begin{tabular}[m]{c|c|c|c|c}
 \hline
$G$ & $H_2(G)$ & $\text{\rm exp}(H_2(G))$ & $\sigma_0(G)$ & $|\text{\rm Idem}(G)|$\\
\hline
%%%%%%%%%%%%%%%%%%%%%%%%%%%%%%%%%%
\multicolumn{5}{c}{Cyclic groups of prime order}\\ \hline
%%%%%%%%%%%%%%%%%%%%%%%%%%%%%%%%%%
 \fbox{$ {\mathbf Z}/p$} & $0$ & $1$ & $1$ & $2$\\ \hline
 %%%%%%%%%%%%%%%%%%%%%%%%%%%%%%%%%%
\multicolumn{5}{c}{Alternating groups}\\ \hline
%%%%%%%%%%%%%%%%%%%%%%%%%%%%%%%%%%
\parbox[m]{3.7cm}{ \center{\fbox{$A_n$, $n\geq 5$} \\ $n\not=6$, $n\not=7$}}&  ${\mathbf Z}/2$ & $2$ & $2$ & $3$\\
\hline
\parbox[m]{3.7cm}{\center{\fbox{$A_6$, $A_7$}} }&${\mathbf Z}/6$ & $6$ & $4$ & $5$\\
\hline
%%%%%%%%%%%%%%%%%%%%%%%%%%%%%%%%%%
\multicolumn{5}{c}{Linear groups}\\ \hline
%%%%%%%%%%%%%%%%%%%%%%%%%%%%%%%%%%
\parbox[m]{3cm}{\center{\fbox{\parbox[m]{2.6cm}{\center{$\text{A}_n(q)$,
 $n\geq 1$ \\  $(n,q)\not=(1,2)$ $(n,q)\not=(1,3)$}}}\\
 $(n,q)\not=(1,4)$ $(n,q)\not=(1,9)$ $(n,q)\not=(2,2)$ $(n,q)\not=(2,4)$ $(n,q)\not=(3,2)$}}&
 ${\mathbf Z}/(n+1,q-1)$ & $(n+1,q-1)$ & $\sigma_0(n+1,q-1)$& $\sigma_0+1$\\
\hline
\parbox[m]{3.7cm}{\center{\fbox{$\text{A}_3(2)$}}}&
 ${\mathbf Z}/2$ & $2$ & $2$& $3$\\
\hline
\parbox[m]{3cm}{\center{$\bullet$ {\fbox{$\text{A}_2(4)$}}} $\bullet$}&
 ${\mathbf Z}/4\oplus {\mathbf Z}/4\oplus {\mathbf Z}/3$ & $12$ & $6$& $7$\\
\hline
\end{tabular}
%\end{center}

%\nopagebreak

%\begin{center}

\begin{tabular}[m]{c|c|c|c|c}
 \hline
$G$ & $H_2(G)$ & $\text{\rm exp}(H_2(G))$ & $\sigma_0(G)$ & $|\text{\rm Idem}(G)|$\\
\hline
%%%%%%%%%%%%%%%%%%%%%%%%%%%%%%%%%%
\multicolumn{5}{c}{Unitary  groups}\\ \hline
%%%%%%%%%%%%%%%%%%%%%%%%%%%%%%%%%%
\parbox[m]{3cm}{\center{\fbox{\parbox[m]{2.6cm}{\center{$^{2}\text{A}_n(q)$,
 $n\geq 2$ $(n,q)\not=(2,2)$}}}\\
 $(n,q)\not=(3,2)$ $(n,q)\not=(3,3)$ $(n,q)\not=(5,2)$}}&
 ${\mathbf Z}/(n+1,q+1)$ & $(n+1,q+1)$ & $\sigma_0(n+1,q+1)$& $\sigma_0+1$\\
\hline
\parbox[m]{3cm}{\center{\fbox{$^{2}\text{A}_3(2)$}}}&
 ${\mathbf Z}/2$ & $2$ & $2$& $3$ \\
\hline
\parbox[m]{3cm}{\center{$\bullet$ {\fbox{$^2\text{A}_3(3)$}}} $\bullet$}&
 ${\mathbf Z}/4\oplus {\mathbf Z}/3\oplus {\mathbf Z}/3$ & $12$ & $6$& $7$\\
\hline
\parbox[m]{3cm}{\center{$\bullet$  {\fbox{$^2\text{A}_5(2)$}}} $\bullet$}&
 ${\mathbf Z}/2\oplus {\mathbf Z}/2\oplus {\mathbf Z}/3$ & $6$ & $4$& $5$\\
\hline
%%%%%%%%%%%%%%%%%%%%%%%%%%%%%%%%%%
\multicolumn{5}{c}{Orthogonal  groups of type B}\\ \hline
%%%%%%%%%%%%%%%%%%%%%%%%%%%%%%%%%%
\parbox[m]{3cm}{\center{\fbox{\parbox[m]{2.6cm}{\center{$\text{B}_{n}(q)$, $n\geq 2$ $(n,q)\not= (2,2)$}}}\\
$q$ odd,
$(n,q)\not=(3,3)$}}&
 ${\mathbf Z}/2$ & $2$ & $2$& $3$\\ \hline
 \parbox[m]{3cm}{\center{\fbox{$\text{B}_{3}(3)$}} }&
 ${\mathbf Z}/6$ & $6$ & $4$& $5$\\
\hline
\parbox[m]{3cm}{\center{\fbox{\parbox[m]{2.6cm}{\center{$\text{B}_{n}(q)$,  $n\geq 2$\\
$(n,q)\not= (2,2)$}}} \\
$q$ even,
$(n,q)\not= (3,2)$}}
&
 $0$ & $1$ & $1$& $2$\\ \hline
 \parbox[m]{3cm}{\center{\fbox{$\text{B}_{3}(2)$}}}&
 ${\mathbf Z}/2$ & $2$ & $2$& $3$\\
\hline
%%%%%%%%%%%%%%%%%%%%%%%%%%%%%%%%%%
\multicolumn{5}{c}{Suzuki Groups}\\ \hline
%%%%%%%%%%%%%%%%%%%%%%%%%%%%%%%%%%
\parbox[m]{4cm}{\center{\fbox{$^2\text{B}_{2}(2^{2n+1})$, $n\geq 1$} \\ $n> 1$}}&
 $0$ & $1$ & $1$& $2$\\
\hline
\parbox[m]{4cm}{\center{$\bullet$ \fbox{$^2\text{B}_2(8)$}} $\bullet$}&
 ${\mathbf Z}/2\oplus {\mathbf Z}/2$& $2$ & $2$& $3$\\
\hline
%%%%%%%%%%%%%%%%%%%%%%%%%%%%%%%%%%
\multicolumn{5}{c}{Symplectic  groups}\\ \hline
%%%%%%%%%%%%%%%%%%%%%%%%%%%%%%%%%%
\parbox[m]{4cm}{\center{\fbox{$\text{C}_{n}(q)$, $n\geq 3$,
$q$ odd}}}&
 ${\mathbf Z}/2$ & $2$ & $2$& $3$\\
\hline
%%%%%%%%%%%%%%%%%%%%%%%%%%%%%%%%%%
\multicolumn{5}{c}{Orthogonal  groups of type D}\\ \hline
%%%%%%%%%%%%%%%%%%%%%%%%%%%%%%%%%%
\parbox[m]{4cm}{\center{\fbox{$\text{D}_{n}(q)$, $n\geq 4$}\\
$q$ even, $(n,q)\not=(4,2)$}}&
 $0$ & $1$ & $1$& $2$\\ \hline
 \parbox[m]{4cm}{\center{$\bullet$ \fbox{$\text{D}_{4}(2)$}} $\bullet$}&
 ${\mathbf Z}/2\oplus {\mathbf Z}/2$ & $2$ & $2$& $3$\\ \hline
\parbox[m]{4cm}{\center{\fbox{$\text{D}_{n}(q)$, $n\geq 5$} \\
$q$ odd, $n$ odd }}&
 ${\mathbf Z}/4$ & $4$ & $3$& $4$\\ \hline
 \end{tabular}
%\end{center}

%\nopagebreak

%\begin{center}

\begin{tabular}[m]{c|c|c|c|c}
 \hline
$G$ & $H_2(G)$ & $\text{\rm exp}(H_2(G))$ & $\sigma_0(G)$ & $|\text{\rm Idem}(G)|$\\
\hline
%%%%%%%%%%%%%%%%%%%%%%%%%%%%%%%%%%
\multicolumn{5}{c}{Orthogonal  groups of type D}\\ \hline
%%%%%%%%%%%%%%%%%%%%%%%%%%%%%%%%%%
 \parbox[m]{4cm}{ \center{$\bullet$$\bullet$ \fbox{$\text{D}_{n}(q)$, $n\geq 4$} $\bullet$$\bullet$\\
$q$ odd, $n$ even }}&
 ${\mathbf Z}/2\oplus {\mathbf Z}/2$ & $2$ & $2$& $4$\\ \hline
\parbox[m]{4cm}{\center{\fbox{$^2\text{D}_{n}(q)$, $n\geq 4$} \\
$q$ even}}&
 $0$ & $1$ & $1$& $2$\\ \hline
\parbox[m]{4cm}{\center{\fbox{$^2\text{D}_{n}(q)$, $n\geq 4$}\\
$q$ odd}}&
 ${\mathbf Z}/2$ & $2$ & $2$& $3$\\ \hline
 %%%%%%%%%%%%%%%%%%%%%%%%%%%%%%%%%%
\multicolumn{5}{c}{Exceptional groups of Lie type}\\ \hline
%%%%%%%%%%%%%%%%%%%%%%%%%%%%%%%%%%
\parbox[m]{4cm}{\center{\fbox{$\text{G}_{2}(q)$, $q\not=2$} \\
$q\not= 3$, $q\not= 4$}}&
 $0$ & $1$ & $1$& $2$\\ \hline
\parbox[m]{4cm}{\center{\fbox{$\text{G}_{2}(3)$}}}&
 ${\mathbf Z}/3$  & $3$ & $2$& $3$\\ \hline
 \parbox[m]{4cm}{\center{\fbox{$\text{G}_{2}(4)$}}}&
 ${\mathbf Z}/2$  & $2$ & $2$& $3$\\ \hline
  \parbox[m]{4cm}{\center{\fbox{$^2\text{G}_{2}(3^{2n+1})$, $n\geq 1$}}}&
 $0$  & $1$ & $1$& $2$\\ \hline
   \parbox[m]{4cm}{\center{\fbox{$\text{F}_{4}(q)$} \\
  $q\not=2$}}&
 $0$  & $1$ & $1$& $2$\\ \hline
   \parbox[m]{4cm}{\center{\fbox{$\text{F}_{4}(2)$}}}&
  ${\mathbf Z}/2$  & $2$ & $2$& $3$\\ \hline
    \parbox[m]{4cm}{\center{\fbox{$^{2}\text{F}_{2}(2^{2n+1})$, $n\geq 1$}}}&
 $0$  & $1$ & $1$& $2$\\ \hline
    \parbox[m]{4cm}{\center{
  \fbox{$^{2}\text{F}_{2}(2)'$}}}&
 $0$  & $1$ & $1$& $2$\\ \hline
      \parbox[m]{4cm}{\center{
  \fbox{$\text{E}_{6}(q)$} \\
  $q\not\equiv 1\text{ mod } 3$}}&
 $0$  & $1$ & $1$& $2$\\ \hline
      \parbox[m]{4cm}{\center{
  \fbox{$\text{E}_{6}(q)$}\\
  $q\equiv 1\text{ mod } 3$}}&
 ${\mathbf Z}/3$  & $3$ & $2$& $3$\\ \hline
       \parbox[m]{4cm}{\center{\fbox{$^2\text{E}_{6}(q)$}\\ $q\not\equiv -1\text{ mod } 3$}}&
 $0$  & $1$ & $1$& $2$\\ \hline
      \parbox[m]{4cm}{\center{\fbox{$^2\text{E}_{6}(q)$} \\ $q\equiv -1\text{ mod } 3$, $q\not= 2$}}&
 ${\mathbf Z}/3$  & $3$ & $2$& $3$\\ \hline
      \parbox[m]{4cm}{\center{$\bullet$ \fbox{$^2\text{E}_{6}(2)$} $\bullet$}}&
 ${\mathbf Z}/2\oplus {\mathbf Z}/2\oplus {\mathbf Z}/3$  & $6$ & $4$& $5$\\ \hline
        \parbox[m]{4cm}{\center{\fbox{$\text{E}_{7}(q)$}\\
      $q$ even}}&
 $0$  & $1$ & $1$& $2$\\ \hline
       \parbox[m]{4cm}{\center{\fbox{$\text{E}_{7}(q)$}\\
      $q$ odd}}&
 ${\mathbf Z}/2$  & $ 2$ & $2$& $3$\\ \hline
         \parbox[m]{4cm}{\center{\fbox{$\text{E}_{8}(q)$}}}&
 $0$  & $1$ & $1$& $2$\\ \hline
  \end{tabular}
\end{center}

\begin{center}
\begin{tabular}[m]{c|c|c|c|c}
 \hline
 $G$ & $H_2(G)$ & $\text{\rm exp}(H_2(G))$ & $\sigma_0(G)$ & $|\text{\rm Idem}(G)|$\\ \hline
   %%%%%%%%%%%%%%%%%%%%%%%%%%%%%%%%%%
\multicolumn{5}{c}{Sporadic groups}\\ \hline
%%%%%%%%%%%%%%%%%%%%%%%%%%%%%%%%%%
       \parbox[m]{4cm}{\center{\fbox{$M_{11}$}}}&
 $0$  & $1$ & $1$& $2$\\ \hline
        \parbox[m]{4cm}{\center{\fbox{$M_{12}$}}}&
 ${\mathbf Z}/2$  & $2$ & $2$& $3$\\ \hline
        \parbox[m]{4cm}{\center{\fbox{$M_{22}$}}}&
 ${\mathbf Z}/12$  & $12$ & $6$& $7$\\ \hline
        \parbox[m]{4cm}{\center{\fbox{$M_{23}$, $M_{24}$, $J_{1}$}}}&
 $0$  & $1$ & $1$& $2$\\ \hline
        \parbox[m]{4cm}{\center{\fbox{$J_{2}$}}}&
 ${\mathbf Z}/2$  & $2$ & $2$& $3$\\ \hline
        \parbox[m]{4cm}{\center{\fbox{$J_{3}$}}}&
 ${\mathbf Z}/3$  & $3$ & $2$& $3$\\ \hline
         \parbox[m]{4cm}{\center{\fbox{$J_{4}$}}}&
 $0$  & $1$ & $1$& $2$\\ \hline
         \parbox[m]{4cm}{\center{\fbox{HS}}}&
 ${\mathbf Z}/2$  & $2$ & $2$& $3$\\ \hline
          \parbox[m]{4cm}{\center{\fbox{He}}}&
 $0$  & $1$ & $1$& $2$\\ \hline
        \parbox[m]{4cm}{\center{\fbox{Mc}}}&
 ${\mathbf Z}/3$  & $3$ & $2$& $3$\\ \hline
         \parbox[m]{4cm}{\center{\fbox{Suz}}}&
 ${\mathbf Z}/6$  & $6$ & $4$& $5$\\ \hline
           \parbox[m]{4cm}{\center{\fbox{Ly}}}&
 $0$  & $1$ & $1$& $2$\\ \hline
         \parbox[m]{4cm}{\center{\fbox{Ru}}}&
 ${\mathbf Z}/2$  & $2$ & $2$& $3$\\ \hline
        \parbox[m]{4cm}{\center{\fbox{O'N}}}&
 ${\mathbf Z}/3$  & $3$ & $2$& $3$\\ \hline
        \parbox[m]{4cm}{\center{\fbox{$\text{Co}_1$}}}&
 ${\mathbf Z}/2$  & $2$ & $2$& $3$\\ \hline
           \parbox[m]{4cm}{\center{\fbox{$\text{Co}_2$, $\text{Co}_3$}}}&
 $0$  & $1$ & $1$& $2$\\ \hline
       \parbox[m]{4cm}{\center{\fbox{$\text{Fi}_{22}$}}}&
 ${\mathbf Z}/6$  & $6$ & $4$& $5$\\ \hline
        \parbox[m]{4cm}{\center{\fbox{$\text{Fi}_{23}$}}}&
 $0$  & $1$ & $1$& $2$\\ \hline
       \parbox[m]{4cm}{\center{\fbox{$\text{Fi}_{24}'$}}}&
 ${\mathbf Z}/3$  & $3$ & $2$& $3$\\ \hline
        \parbox[m]{4cm}{\center{\fbox{$\text{F}_{5}$, $\text{F}_{3}$}}}&
 $0$  & $1$ & $1$& $2$\\ \hline
        \parbox[m]{4cm}{\center{\fbox{$\text{F}_{2}$}}}&
 ${\mathbf Z}/2$  & $2$ & $2$& $3$\\ \hline
        \parbox[m]{4cm}{\center{\fbox{$\text{F}_{1}$}}}&
 $0$  & $1$ & $1$& $2$\\ \hline
 \end{tabular}
\end{center}
 \end{prop}
 \begin{proof}
 The proposition can be derived by elementary arguments from~\cite[6.3.1]{GLS3}. For self containment we
 present these elementary arguments below.

 We need to explain the table only for the groups whose Schur multiplier is not cyclic. There are just  seven such cases.
 \smallskip

 \noindent
 {\bf Cases:} $^2\text{\rm A}_5(2)$, $^2\text{\rm B}_2(8)$, $\text{\rm D}_{4}(2)$, $^2\text{\rm E}_{6}(2)$.
 In all of these cases the Schur multiplier is of the form ${\mathbf Z}/2\oplus {\mathbf Z}/2$ or
 ${\mathbf Z}/2\oplus {\mathbf Z}/2 \oplus  {\mathbf Z}/3$.  To prove the proposition we need to find  all the invariant
 subgroups
 of the 2-torsion part. In all of these cases, according
 to~\cite[6.3.1]{GLS3}, the group  $\text{Out}(G)$ contains ${\mathbf Z}/3$ which acts faithfully on the $2$-torsion part ${\mathbf Z}/2\oplus {\mathbf Z}/2$.  However if $\psi:{\mathbf Z}/2\oplus {\mathbf Z}/2\ra {\mathbf Z}/2\oplus {\mathbf Z}/2$
 is an automorphism such that $\psi\not=\text{id}$ and $\psi^3=\text{id}$, then
$\psi$ has no eigenvectors. Consequently  the only subgroups of  ${\mathbf Z}/2\oplus {\mathbf Z}/2$
which are invariant under $\psi$ are the trivial subgroup and  the whole group.
\medskip

\noindent
{\bf Case:} $\text{\rm A}_2(4)$. In this case the Schur multiplier is isomorphic to
${\mathbf Z}/4\oplus {\mathbf Z}/4 \oplus  {\mathbf Z}/3$. To prove the proposition we need to understand
the invariant subgroups of the $4$-torsion part $V:={\mathbf Z}/4\oplus {\mathbf Z}/4$.  Again
according to~\cite[6.3.1]{GLS3}, the group  $\text{Out}(\text{A}_2(4))$ contains ${\mathbf Z}/3$ which acts faithfully on
${\mathbf Z}/4\oplus {\mathbf Z}/4$. Let  $\psi: V\ra V$ be an automorphism of order $3$.  We claim that the only
 $\psi$-invariant subgroups in $V$ are the trivial subgroup, the
Frattini subgroup $\Phi\cong {\mathbf Z}/2\oplus {\mathbf Z}/2$, and the whole group $V$.
We have $V=[V,\psi]\oplus C_V(\psi)$ and so if $C_V(\psi)$ is non-trivial,
then $[V,\psi]\cong {\mathbf Z}/4$ contradicting the fact that $\psi$
is faithful on $[V,\psi]$.  Hence $C_V(\psi)$ is trivial.
Let $K$ be a $\psi$-invariant subgroup.  If $|K|=2$ or $4$ and $K\ne\Phi,$
then $K$ is cyclic and hence it is centralized by $\psi$, a contradiction.
If $|K|=8$, then $K\cong {\mathbf Z}/4\oplus {\mathbf Z}/2$, so $\psi$ centralizes
the Frattini subgroup of $K$, a contradiction.
\medskip

 \noindent
 {\bf Case:} $^2\text{\rm A}_3(3)$. In this case the Schur multiplier is isomorphic to
  ${\mathbf Z}/4\oplus {\mathbf Z}/3\oplus {\mathbf Z}/3$. To prove the proposition we need to
  understand the invariant subgroup of the $3$-torsion part. According to~\cite[6.3.1]{GLS3}, the group  $\text{Out}(^2\text{A}_3(3))$ contains ${\mathbf Z}/4$ which acts faithfully on  $V:={\mathbf Z}/3\oplus {\mathbf Z}/3$.
Let $\psi: V\ra V$ be an
automorphism of order $4$ acting faithfully.  If there would be a proper non-trivial $\psi$-invariant
subspace $W\subset V$, then $V$ would split as a direct sum $V=W\oplus U$ with $U$
$\psi$-invariant.  But then $\psi^2$ would centralize $V$, a contradiction.
Hence the only $\psi$-invariant subspaces are the trivial one and $V$.
 \medskip

 \noindent
 {\bf Case:} $\text{\rm D}_{n}(q)$, $n\geq 4$,
$q$ odd, $n$ even. In this case the Schur multiplier is isomorphic to ${\mathbf Z}/2\oplus {\mathbf Z}/2$.
According to~\cite[6.3.1]{GLS3}, after an appropriate choice of a base,   automorphisms of  $\text{\rm D}_{n}(q)$ act
on the Schur multiplier  ${\mathbf Z}/2\oplus {\mathbf Z}/2$ either as  the identity or the transposition.  Moreover there is an element that does act as a transposition.
It follows that, with this choice of a base, the invariant subgroups are: the trivial subgroup, the diagonal, and  the whole group.
 \end{proof}
%%%%%%%%%%%%%%%%%%%%%%%%%%%%%%%%%%%%%%%%%%%%%%%%%%%%%%%
%%%%%%%%%%%%%%%%%%%%%%%%%%%%%%%%
%%%%%%%%%%%%%%%%%%%%%%%%%%%%%%%

\end{document}